\RequirePackage{silence} 
\documentclass[CRMATH,Unicode,manuscript]{cedram}
\newcommand{\R}{{\mathbb R}}
\renewcommand{\S}{{\mathbb S}}
\newcommand{\N}{{\mathbb N}}

\newcommand{\be}[1]{\begin{equation}\label{#1}}
\newcommand{\ee}{\end{equation}}
\renewcommand{\(}{\left(}
\renewcommand{\)}{\right)}
\newcommand{\isd}[1]{\int_{\S^d}{#1}\,d\mu_d}
\newcommand{\nrmsd}[2]{\|{#1}\|_{\mathrm L^{#2}(\S^d,d\mu_d)}}

\newcommand{\ird}[1]{\int_{\R^n}{#1}\,dx}
\newcommand{\irn}[1]{\int_{\R^n}{#1}\,dy}
\newcommand{\irng}[1]{\int_{\R^n}{#1}\,d\gamma}
\newcommand{\nrmng}[2]{\|#1\|_{\mathrm L^{#2}(\R^n,d\gamma)}}
\newcommand{\nrmnmu}[2]{\|#1\|_{\mathrm L^{#2}(\R^n,d\mu)}}
\newcommand{\Nrmnmu}[2]{\big\|#1\big\|_{\mathrm L^{#2}(\R^n,d\mu)}}
\newcommand{\cb}[1]{\langle{#1}\rangle}
\newcommand{\msc}[1]{\href{https://mathscinet.ams.org/mathscinet/search/mscbrowse.html?sk=default&sk=#1&submit=Search}{#1}}
\newcommand{\tf}{\tfrac1d}

\newcommand{\irnmu}[1]{\int_{\R^n}{#1}\,d\mu}
\newcommand{\nrms}[2]{\left\|#1\right\|_{\mathrm L^{#2}(\S^d)}}
\newcommand{\irneps}[1]{\int_{B_\varepsilon}{#1}\,d\gamma}

\newcommand{\patch}{\if{}\else{}}
\title[On Gaussian interpolation inequalities]{On Gaussian interpolation inequalities}
\alttitle{Sur des in\'egalit\'es d'interpolation Gaussiennes}

\author{\firstname{Giovanni} \lastname{Brigati}\CDRorcid{0000-0002-3698-9995}}
\address{CEREMADE (CNRS UMR n$^\circ$~7534), PSL University, Universit\'e Paris-Dauphine,\newline Place de Lattre de Tassigny, 75775 Paris 16, France}
\email[G.~Brigati]{brigati@ceremade.dauphine.fr}
\author{\firstname{Jean} \lastname{Dolbeault}\CDRorcid{0000-0003-4234-2298}}
\address{CEREMADE (CNRS UMR n$^\circ$~7534), PSL University, Universit\'e Paris-Dauphine,\newline Place de Lattre de Tassigny, 75775 Paris 16, France}
\email[J.~Dolbeault]{dolbeaul@ceremade.dauphine.fr}
\author{\firstname{Nikita} \lastname{Simonov}\CDRorcid{0000-0002-3241-190X}}
\address{LJLL (CNRS UMR n$^\circ$~7598) Sorbonne Universit\'e,\newline 4 place Jussieu, 75005 Paris, France}
\email[N.~Simonov]{nikita.simonov@sorbonne-universite.fr}


\begin{abstract} 
This paper is devoted to Gaussian interpolation inequalities with endpoint cases corresponding to the Gaussian Poincar\'e and the logarithmic Sobolev inequalities, seen as limits in large dimensions of Gagliardo-Nirenberg-Sobolev inequalities on spheres. Entropy methods are investigated using not only heat flow techniques but also nonlinear diffusion equations as on spheres. A new stability result is established for the Gaussian measure, which is directly inspired by recent results for spheres.
\end{abstract}

\begin{altabstract} 
Cet article est consacr\'e \`a des in\'egalit\'es d'interpolation Gaussiennes, avec comme cas extrêmes l'in\'egalit\'e de Poincar\'e Gaussienne et l'in\'egalit\'e de Sobolev logarithmique, vues comme limites en grandes dimensions des in\'egalit\'es de Gagliardo-Nirenberg-Sobolev sur les sph\`eres. Les m\'ethodes d'entropie sont abord\'ees en utilisant non seulement des techniques bas\'ee sur l'\'equation de la chaleur mais aussi sur des \'equations de diffusion non-lin\'eaires, comme pour les sph\`eres. Un nouveau r\'esultat de stabilit\'e est \'etabli pour les mesures Gaussiennes, qui s'inspire directement de r\'esultats r\'ecents sur les sph\`eres.
\end{altabstract}

\begin{document}
\maketitle

%%%%%%%%%%%%%%%%%%%%%%%%%%%%%%%%%%%%%%%%%%%%%%%%%%%%%%%%%%%%%%%%%%%%%%%%%%
%%%%%%%%%%%%%%%%%%%%%%%%%%%%%%%%%%%%%%%%%%%%%%%%%%%%%%%%%%%%%%%%%%%%%%%%%%
\section{Introduction and main results}\label{Sec:Intro}

Let us consider the \emph{Gagliardo-Nirenberg-Sobolev inequalities} on the unit $d$-dimensional sphere
\be{Ineq:GNS}
\nrmsd{\nabla u}2^2\ge\frac d{p-2}\(\nrmsd up^2-\nrmsd u2^2\)\quad\forall\,u\in\mathrm H^1(\S^d,d\mu_d)
\ee
for any $p\in[1,2)\cup(2,+\infty)$ if $d=1$, $2$, and for any $p\in[1,2)\cup(2,2^*]$ if $d\ge3$. Here $d\mu_d$ denotes the uniform probability measure on $\S^d\subset\R^{d+1}$ and, if $d\ge3$, $2^*=2\,d/(d-2)$ is the critical Sobolev exponent. By convention, we take $2^*=+\infty$ if $d=1$ or $2$. The purpose of this paper is to clarify the links of these interpolation inequalities with the family of \emph{Gaussian interpolation inequalities}
\be{Ineq:PLS}
\nrmng{\nabla v}2^2\ge\frac1{2-p}\(\nrmng v2^2-\nrmng vp^2\)\quad\forall\,v\in\mathrm H^1(\R^n,d\gamma)
\patch\ee
where the exponent $p$ is taken in the range $1\le p<2$. Inequality~\eqref{Ineq:PLS} is intermediate between the Poincar\'e inequality corresponding to $p=1$ and the \emph{Gaussian logarithmic Sobolev inequality}
\[
\nrmng{\nabla v}2^2\ge\frac12\irng{|v|^2\,\log\(\frac{|v|^2}{\nrmng v2^2}\)}\quad\forall\,v\in\mathrm H^1(\R^n,d\gamma)
\]
obtained as a limit case of~\eqref{Ineq:PLS} as $p\to2_-$. Here $d\gamma(y):=(2\,\pi)^{-n/2}\,e^{-\frac12\,|y|^2}\,dy$ denotes the centred normalized Gaussian probability measure and the dimension $n$ is any positive integer.

\medskip It is somewhat classical that, if we consider the Sobolev inequality on the sphere, \emph{i.e.}, Inequality~\eqref{Ineq:GNS} with $p=2^*$ and $d\ge3$, rescale it to the sphere of radius $\sqrt d$ and fix a function depending on $n$ variables only on this sphere, since the curvature tends to $0$ at the right order as one takes the limit as $d\to+\infty$, then the sphere tends to the flat space with Gaussian measure. For instance, we read (with adapted notation) from~\cite[p.~4818]{MR1164616} that \emph{if we rescale this inequality so as to be on a sphere of radius $\sqrt d$ and take the limit $d\to\infty$ or $p\to2$ we obtain in the Poincar\'e limit the Gross logarithmic inequality for the Gaussian measure since $-(1/d)\,\Delta$ on $\S^d$ goes to $-\Delta+ x\cdot\nabla$ on the infinite-dimensional limit.} The last part of the sentence refers to a result known as the Maxwell–Poincar\'e lemma : see~\cite{MR353471} and~\cite[Remark~4, p.~254]{Vershik:2007nr} for some historical comments. The statement of~\cite{MR1164616} has been made more precise later in~\cite{MR1479040,MR1673903} using a slightly different limit. However, to our knowledge, the infinite-dimensional limit has not been considered in the subcritical range $p<2^*$. 

On the sphere, Inequality~\eqref{Ineq:GNS} follows from~\cite{MR889476,MR808640} for any $p\ge1$ if $d=1$ and any $p\in[1,2^\#]$ with $2^\#:=(2\,d^2+1)/(d-1)^2$ if $d\ge2$. The proof in the range $p\in(2^\#,+\infty)$ if $d=2$ and $p\in(2^\#,2^*]$ if $d\ge3$ can be found in~\cite[Corollary~6.1]{MR1134481},~\cite{MR1213110} and~\cite{MR1230930}. Also see~\cite{MR717827} in the case $p=2^*$. In the case of the Gaussian measure, we refer to~\cite{MR954373} (also see~\cite{MR1796718}) for a first proof of Inequality~\eqref{Ineq:PLS}. The formal analogy of~\eqref{Ineq:GNS} and~\eqref{Ineq:PLS} is striking. Although computations are somewhat standard, our first purpose is to make this point rigorous and recover~\eqref{Ineq:PLS} as a special limit of~\eqref{Ineq:GNS} as $d\to+\infty$.
%-------------------------------------------------------------------------
\begin{theorem}\label{Thm:GNStoPLS} Let $n$ be a positive integer, $p\in[1,2)$ and consider a function $v\in\mathrm H^1(\mathrm R^n,dx)$ with compact support. For any $d\ge n$, large enough, if $u_d(\omega)=v\big(\omega_1/\sqrt d,\omega_2/\sqrt d,\ldots,\omega_n/\sqrt d\big)$ where $\omega=(\omega_1,\omega_2,\ldots,\omega_{d+1})\in\S^d\subset\R^{d+1}$, then
\begin{multline*}
\lim_{d\to+\infty}d\(\nrmsd{\nabla u_d}2^2-\frac d{2-p}\(\nrmsd{u_d}2^2-\nrmsd{u_d}p^2\)\)\\
=\nrmng{\nabla v}2^2-\frac1{2-p}\(\nrmng v2^2-\nrmng vp^2\)\,.
\end{multline*}
\end{theorem}
%-------------------------------------------------------------------------
\medskip The \emph{carr\'e du champ} method has frequently been applied to prove Gaussian interpolation inequalities ranging between the logarithmic Sobolev inequality and the Poincar\'e inequality like~\eqref{Ineq:PLS} using the linear flow associated with the Ornstein-Uhlenbeck operator; see~\cite{MR889476}, and~\cite{MR3155209} for an overview. Still in the case of the Gaussian measure, we adopt here a new point of view by using \emph{nonlinear diffusion equations} in order to prove the same inequalities, but with different remainder terms. This is a very natural point of view when dealing with inequalities like~\eqref{Ineq:GNS} on the sphere, as shown for instance in~\cite{MR3229793} (see earlier references therein). In that case, linear flows are indeed limited to exponents $p\le2^\#$ if $d\ge2$. To overcome this difficulty if either $d=2$ and $p>2^\#$ or $d\ge3$ and $p\in(2^\#,2^*]$, one has to consider fast diffusion flows. Before explaining the results for the Gaussian measure, let us summarize the main known results on the sphere.

On $\S^d$, let us consider a positive solution $u$ of
\be{FDE}
\frac{\partial u}{\partial t}=u^{-p\,(1-m)}\(\Delta u+(m\,p-1)\,\frac{|\nabla u|^2}u\)
\ee
where $\Delta$ denotes the Laplace-Beltrami operator on $\S^d$. In the case $m=1$, $u^p$ solves the heat equation and for this reason, we shall call it the \emph{linear case}. Otherwise $u^p$ solves a nonlinear diffusion equation corresponding either to a fast diffusion flow with $m<1$ or to a porous media equation with $m>1$. In any case, we claim that
\[
\frac d{dt}\nrms up^2=0\quad\mbox{and}\quad\frac d{dt}\nrms u2^2=2\,(p-2)\isd{u^{-\,p\,(1-m)}\,|\nabla u|^2}\,.
\]
Let us define
\be{admissible.m}
m_\pm(d,p):=\frac1{(d+2)\,p}\(d\,p+2\pm\sqrt{d\,(p-1)\,\big(2\,d-(d-2)\,p\big)}\)\,.
\ee
The following result can be found in~\cite{MR2381156,MR3229793} with additional details in~\cite{MR3640894,MR4095470,BDS2022-Sphere}.
%-------------------------------------------------------------------------
\begin{proposition}[\hspace*{-3pt}\cite{MR2381156,MR3229793}]\label{Prop:BE-Sphere} Assume that $d\ge1$, with either $p\in[1,2)\cup(2,+\infty)$ if $d=2$ or $p\in[1,2)\cup(2,2^*]$ if $d\ge3$, and let $m\in[m_-(d,p),m_+(d,p)]$. If $u>0$ solves~\eqref{FDE} with an initial datum in $\mathrm L^2\cap\mathrm L^p(\S^d,d\mu_d)$, then
\[
\frac d{dt}\(\nrmsd{\nabla u}2^2-\frac d{p-2}\(\nrmsd up^2-\nrmsd u2^2\)\)\le0\quad\forall\,t>0\,.
\]
\end{proposition}
%-------------------------------------------------------------------------
\noindent The limit, as $t\to+\infty$, of any solution of~\eqref{FDE} is a constant. This means that the \emph{deficit}, that is, the difference of the two sides in Inequality~\eqref{Ineq:GNS}, converges to $0$. Then, by Proposition~\ref{Prop:BE-Sphere}, it follows that the deficit is non-negative, which directly proves~\eqref{Ineq:GNS}. The same monotonicity property applies to the deficit of the \emph{logarithmic Sobolev inequality on the sphere}
\[
\nrmsd{\nabla u}2^2\ge\frac d2\isd{|u|^2\,\log\(\frac{|u|^2}{\nrms u2^2}\)}\quad\forall\,u\in\mathrm H^1(\S^d,d\mu_d)
\]
in the limit case corresponding to $p=2$. The admissible values of the parameters are limited to $m_-(d,p)\le m\le m_+(d,p)$ and $1\le p\le2^*$ if $d\ge3$. Moreover, at the endpoints, we have $m_\pm(d,1)=1$ and $m_\pm(d,2^*)=(d-1)/d$ if $d\ge3$, while $m_+(d,2^\#)=1$ so that $m=1$ is admissible if and only if $1\le p\le2^\#$ when $d\ge2$. See Fig.~\ref{F1}. For appropriate initial data, it is shown in~\cite{MR3640894} that the monotonicity property of the deficit along the flow of~\eqref{FDE} is violated for any $p\in[2,2^*)$ or $p=2^*$ if $d\ge3$ as soon as either $m<m_-(d,p)$ or $m>m_+(d,p)$. The case $p=2$ corresponding to the logarithmic Sobolev inequality is included.
%-------------------------------------------------------------------------
\begin{figure}[ht]
\includegraphics[width=6cm]{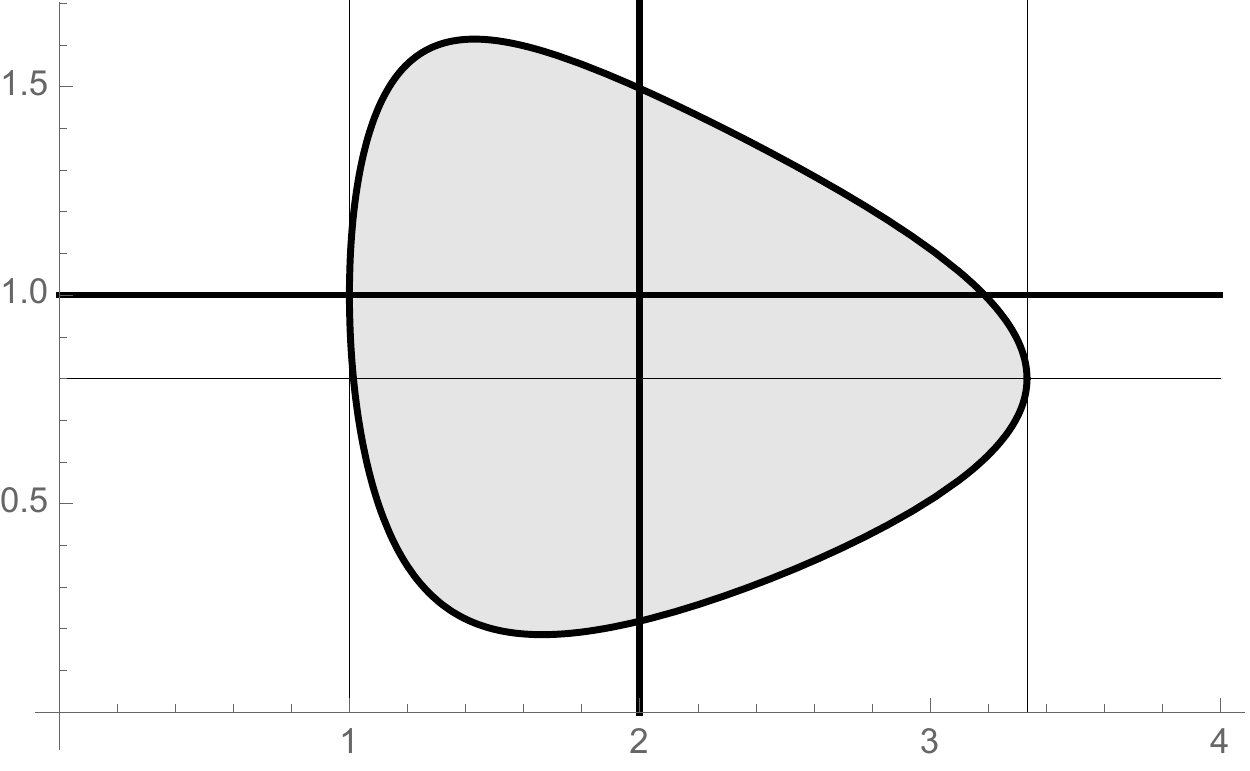}
\caption{\label{F1} Case $d=5$. The admissible parameters $p$ and $m$ correspond to the grey area. The boundary of the admissible set is tangent to the vertical lines $p=1$ at $(m,p)=(1,1)$ and $p=2^*=10/3$ at $(m,p)=(4/5,10/3)$. Qualitatively, this figure does not change as $d$ increases but gets squeezed in the interval $1\le p\le2$ as $d\to+\infty$.}
\end{figure}
%-------------------------------------------------------------------------

In view of the results of Theorem~\ref{Thm:GNStoPLS}, it is natural to ask whether there is also a monotonicity property of the deficit associated to the Gaussian interpolation inequalities~\eqref{Ineq:PLS} when we rely on a nonlinear diffusion flow on $\R^n$. Let $m_\pm(p):=\lim_{d\to+\infty}m_\pm(d,p)$ where $m_\pm$ is given by~\eqref{admissible.m} and notice that
\be{mpm}
m_\pm(p)=1\pm\frac1p\,\sqrt{(p-1)\,(2-p)}\,.
\ee
The diffusion operator associated to the Gaussian measure is the \emph{Ornstein-Uhlenbeck operator} $\mathcal L=\Delta-x\cdot\nabla$ and we consider now the nonlinear parabolic equation
\be{FDEgamma}
\frac{\partial v}{\partial t}=v^{-p\,(1-m)}\(\mathcal Lv+(m\,p-1)\,\frac{|\nabla v|^2}v\)\,.
\ee
In the definition of $\mathcal L$, $\Delta$ denotes the standard Laplacian on $\R^n$. The following result is new for $m\neq1$ while the case $m=1$ follows from the method of the \emph{carr\'e du champ} developed by Bakry and Emery in~\cite{MR889476}.
%-------------------------------------------------------------------------
\begin{theorem}\label{Thm:BE-Gaussian} Assume that $n\ge1$, $p\in[1,2)$. If $v>0$ solves~\eqref{FDEgamma} with $m\in[m_-(p),m_+(p)]$ for an initial datum in $\mathrm L^2\cap\mathrm L^p(\R^n,d\gamma)$, then
\[
\frac d{dt}\(\nrmng{\nabla v}2^2-\frac 1{p-2}\(\nrmng vp^2-\nrmng v2^2\)\)\le0\quad\forall\,t>0\,.
\]
\end{theorem}
%-------------------------------------------------------------------------
\noindent The limiting case $p=2$ corresponding to the Gaussian logarithmic Sobolev inequality is also covered but it is obtained as a standard application of the linear \emph{carr\'e du champ} method known from~\cite{MR889476} because $m_\pm(2)=1$. See Fig.~\ref{F2}.
%-------------------------------------------------------------------------
\begin{figure}[ht]
\includegraphics[width=6cm]{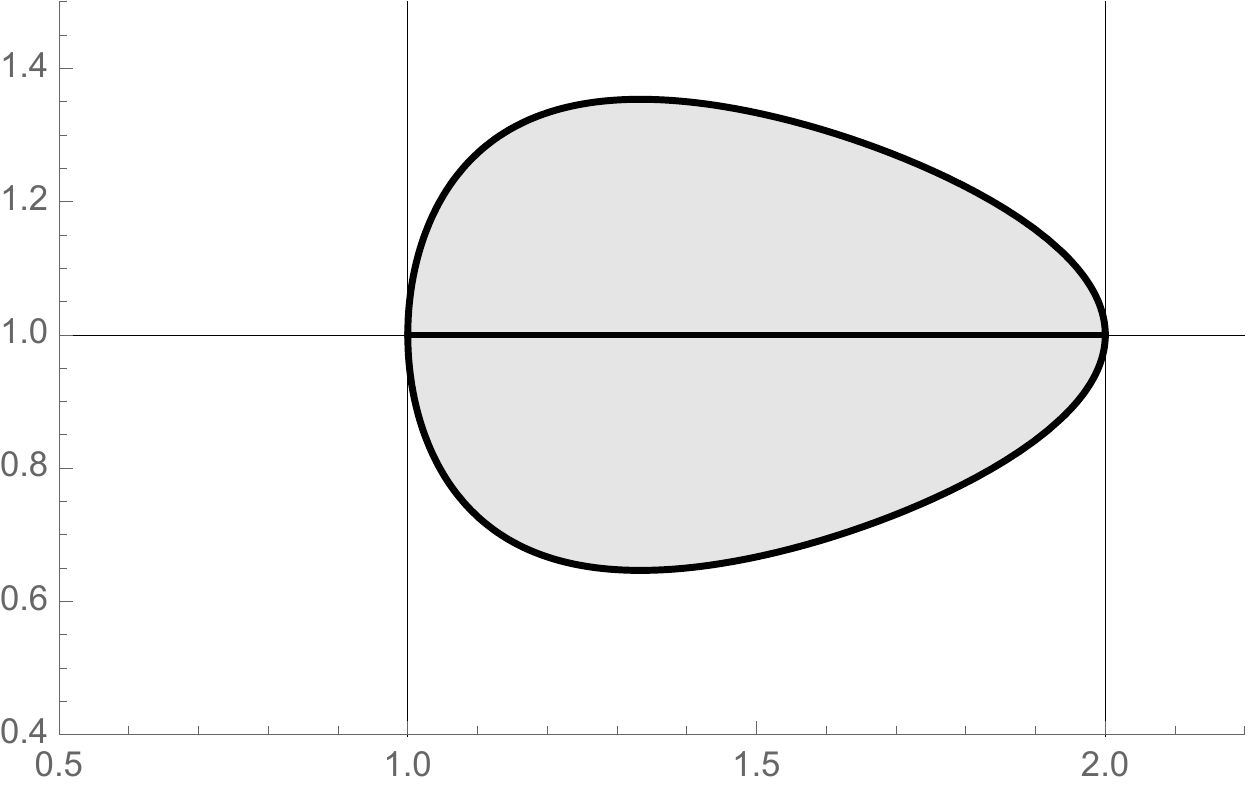}
\caption{\label{F2} The admissible parameters $p$ and $m$ correspond to the grey area and are independent of the dimension $n$. The boundary of the admissible set is tangent to the vertical lines $p=1$ at $(m,p)=(1,1)$ and $p=2$ at $(m,p)=(1,2)$. It is the limit set of the admissible parameters for Proposition~\ref{Prop:BE-Sphere} as $d\to+\infty$.}
\end{figure}
%-------------------------------------------------------------------------

Said in simple words, the result of Theorem~\ref{Thm:BE-Gaussian} is that the admissible range of exponents of the nonlinear flow, for which the deficit associated to~\eqref{Ineq:PLS} is monotone non-increasing, is obtained as the limit of the range of the corresponding exponents on the sphere, in the large dimensions limit. Moreover, $p=2$ appears as a \emph{critical exponent} for the Gaussian measure.

\medskip Let us now focus on stability results. The main result of~\cite{BDS2022-Sphere} for the sphere is a constructive stability estimate for Inequality~\eqref{Ineq:GNS}, limited to the subcritical range $p\in(1,2^*)$, which measures the distance to optimal functions, and distinguishes the subspaces generated by constant functions, spherical harmonic functions associated to the first positive eigenvalue of the Laplace-Beltrami operator, and the orthogonal directions. Optimal exponents in the stability estimate measuring the distance to the set of optimal functions differ, depending on the directions. Here we have the exact counterpart in the Gaussian case. Let~$\Pi_1$ denote the orthogonal projection of $\mathrm L^2(\R^n,d\gamma)$ onto the $(n+1)$-dimensional function space generated by $1$ and $x_i$ with $i=1$, $2$,\,\ldots,$n$.
%-------------------------------------------------------------------------
\begin{theorem}\label{stability-subcritical-Gaussian}
For all $n\ge 1$, and all $p\in(1,2),$ there is an explicit constant $c_{n,p}>0$ such that, for all $v \in \mathrm{H}^1(d\gamma)$, it holds 
\begin{multline*}
\nrmng{\nabla v}2^2-\frac 1{p-2}\(\nrmng vp^2-\nrmng v2^2\)\\
\ge\,c_{n,p}\(\nrmng{\nabla(\mathrm{Id}-\Pi_1)v}2^2+\frac{\nrmng{\nabla\Pi_1v}2^4}{\nrmng{\nabla v}2^2+\nrmng v2^2}\).
\end{multline*}
\end{theorem}
%-------------------------------------------------------------------------
\noindent Exponents $2$ and $4$ which appear in the right-hand side of the inequality are sharp and the constant $c_{n,p}$ has an explicit although complicated expression given in the proof. If $p=1$, the last term drops and the distance to optimal functions is measured only by
\[
\nrmng{\nabla(\mathrm{Id}-\Pi_1)v}2^2\,,
\]
and a decomposition on Hermite polynomials shows that the estimate cannot be improved. If $p=2^*$ and $n\ge3$, the recent stability result on Sobolev's inequality on $\R^n$ of~\cite{DEFFL}, which quantifies the estimate of Bianchi and Egnell in~\cite{MR1124290} can be translated into a stability result for~\eqref{Ineq:GNS} on~$\S^d$, that can be recast in the form of~\cite[Theorem 6]{BDS2022-Sphere}. By a large dimension argument, a stability on the Gaussian logarithmic Sobolev inequality is also shown in~\cite{DEFFL}, although the distance is measured only by an $\mathrm L^2(\R^n,d\gamma)$ norm. Whether a stronger estimate can be obtained in the limiting case $p=2$, eventually under some restriction, is therefore so far an open question.

\medskip This paper is organized as follows. In Section~\ref{Sec:LargeDimensions}, we give a new proof of Inequality~\eqref{Ineq:PLS} as a consequence of Inequality~\eqref{Ineq:GNS} by taking a large dimensions limit, applied to the inequality written for a function depending only on a fixed number $n$ of real variables. To our knowledge, this is new except for the limit case $p=2$ of the logarithmic Sobolev inequality. Section~\ref{Sec:Nonlinear} is devoted to the proof of Theorem~\ref{Thm:BE-Gaussian}: we characterize the nonlinear diffusion flows of porous medium or fast diffusion type such that the deficit is monotone non-increasing and recover the picture known on the sphere in the large dimensions limit. Moreover, by the \emph{carr\'e du champ} method, we establish improved inequalities that provide us with first stability results. The stability result of Theorem~\ref{stability-subcritical-Gaussian} for the Gaussian measure is proved in Section~\ref{Sec:Stability} using a detailed Taylor expansion and the improved inequalities of Section~\ref{Sec:Nonlinear}.

%%%%%%%%%%%%%%%%%%%%%%%%%%%%%%%%%%%%%%%%%%%%%%%%%%%%%%%%%%%%%%%%%%%%%%%%%%
%%%%%%%%%%%%%%%%%%%%%%%%%%%%%%%%%%%%%%%%%%%%%%%%%%%%%%%%%%%%%%%%%%%%%%%%%%
\section{From subcritical interpolation inequalities on the sphere to Gaussian interpolation}\label{Sec:LargeDimensions}

In this section we explain how Inequality~\eqref{Ineq:PLS} can be seen as the limit of Inequality~\eqref{Ineq:GNS} in the large dimensions limit, that is, as $d\to+\infty$, and prove Theorem~\ref{Thm:GNStoPLS}. Comments on the limit case $p=2$ can be found at the end of this section.

The unit sphere $\S^d$ is parametrized in terms of the \emph{stereographic coordinates} by
\[
\omega_j=\frac{2\,x_j}{1+|x|^2}\quad\mbox{if}\quad1\le j\le d\quad\mbox{and}\quad\omega_{d+1}=\frac{1-|x|^2}{1+|x|^2}
\]
where $\omega=(\omega_1,\omega_2,\ldots,\omega_{d+1})$ denote the coordinates in $\R^{d+1}\supset\S^d$ and $x=(x_1,x_2,\ldots,x_d)$ are Cartesian coordinates in $\R^d$. To a function $u$ on $\S^d$, we associate a function $w$ on $\R^d$ using the stereographic projection such that
\[
\(\frac2{1+|x|^2}\)^\frac{d-2}2u(\omega)=w(x)\quad\forall\,x\in\R^d\,.
\]
It is a standard result that
\be{up}
\isd{|u|^p}=2^\frac{\delta(p)}2\,\big|\S^d\big|^{-1}\ird{\cb x^{-\delta(p)}\,|w|^p}
\ee
and that
\[
\isd{|\nabla u|^2}+\tfrac14\,d\,(d-2)\isd{|u|^2}=\big|\S^d\big|^{-1}\ird{|\nabla w|^2}\,.
\]
where $\cb x:=\sqrt{1+|x|^2}$ and $\delta(p):=2\,d-p\,(d-2)$. Using the stereographic projection, Inequality~\eqref{Ineq:GNS} can be written on the Euclidean space $\R^d$ as the weighted interpolation inequality
\be{Ineq:wGNS}
\ird{|\nabla w|^2}+\frac{d\,\delta(p)}{p-2}\ird{\frac{|w|^2}{\cb x^4}}\ge\frac{\mathcal C_{d,p}}{p-2}\(\ird{\frac{|w|^p}{\cb x^{\delta(p)}}}\)^\frac2p\quad\mbox{with}\quad\mathcal C_{d,p}=2^\frac{\delta(p)}p\,d\,\big|\S^d\big|^{1-\frac2p}.
\ee
 See~\cite{MR4095470} for details. Equality is achieved by the Aubin-Talenti function $w_\star(x):=\cb x^{2-d}$. Assume that $d\ge4$. Let us consider $f=w/w_\star$ and notice that the inequality rewritten in terms of $f$ is
\[
\ird{|\nabla f|^2\,w_\star^2}+\frac{4\,d}{p-2}\ird{|f|^2\,w_\star^{2^*}}\ge\frac{\mathcal C_{d,p}}{p-2}\(\ird{|f|^p\,w_\star^{2^*}}\)^\frac2p\,.
\]
Since we are interested in the limit as $d\to+\infty$, we shall consider the case $1\le p<2$, which is admissible for any $d\ge2$.
 
With $g(x)=f\big(x/\sqrt d\big)$, we obtain after changing variables that
\[
\frac14\int_{\R^d}|\nabla g|^2\,\frac{dx}{\(1+\tf\,|x|^2\)^{d-2}}+\frac1{p-2}\int_{\R^d}|g|^2\,\frac{dx}{\(1+\tf\,|x|^2\)^d}\ge\frac{\mathcal C_{d,p}\,d^{d\frac{p-2}{2\,p}}}{4\,d\,(p-2)}\(\int_{\R^d}|g|^p\,\frac{dx}{\(1+\tf\,|x|^2\)^d}\)^\frac2p\,.
\]
Let us assume that $n\ge1$ is a given integer and take $d>\max\{n,3\}$. With $x=(y,z)\in\R^n\times\R^{d-n}\approx\R^d$, we also assume that the function $g$ depends only on $y$. In other words, we write $g=g_d$ where $g_d(y,z)=v(y)$ for some function $v$ defined on $\R^n$, that is, 
\be{Marginals}
g_d(y,z)=v(y)\quad\forall\,(y,z)\in\R^n\times\R^{d-n}
\ee
Here we use a subscript $d$ in order to emphasize that $g_d$ has to be considered as a function on $\R^d$. Let us define
\[
c_d:=(d\,\pi)^\frac d2\,\frac{\Gamma\(d/2\)}{\Gamma(d)}\,.
\]
%-------------------------------------------------------------------------
\begin{lemma}\label{Lem:WGNS-PLS} Let $n$ be a positive integer, $p\in[1,2)$, consider a function $v\in\mathrm H^1(\R^n,dx)$ with compact support and define $g_d$ according to~\eqref{Marginals}. Then we have
\begin{multline*}
\lim_{d\to+\infty}\frac1{4\,c_d}\int_{\R^d}|\nabla g_d|^2\,\frac{dx}{\(1+\tf\,|x|^2\)^{d-2}}\\
+\frac1{2-p}\,\lim_{d\to+\infty}
\frac1{c_d}\(\frac{\mathcal C_{d,p}\,d^{d\frac{p-2}{2\,p}}}{4\,d}\(\int_{\R^d}|g_d|^p\,\frac{dx}{\(1+\tf\,|x|^2\)^d}\)^\frac2p-\int_{\R^d}|g_d|^2\,\frac{dx}{\(1+\tf\,|x|^2\)^d}\)\\
=\nrmng{\nabla v}2^2-\frac1{2-p}\(\nrmng v2^2-\nrmng vp^2\)\,.
\end{multline*}
\end{lemma}
%-------------------------------------------------------------------------
\noindent In other words, we prove that the infinite dimensional limit of~\eqref{Ineq:wGNS}, for functions depending only on a finite number $n$ of real variables, is~\eqref{Ineq:PLS}. The assumption of compact support can be removed if~$v$ and $\nabla v$ have sufficient decay properties at infinity.

\begin{proof} Using
\[
\(1+\tf\,|x|^2\)^{2-d}=\Big(1+\tf\(|y|^2+|z|^2\)\Big)^{2-d}=\(1+\tf\,|y|^2\)^{2-d}\(1+\tf\,|\zeta|^2\)^{2-d}\quad\mbox{with}\quad\zeta=\frac z{\sqrt{1+\tf\,|y|^2}}\,,
\]
we can integrate with respect to $z$ and obtain
\[
\int_{\R^{d-n}}\(1+\tf\,|x|^2\)^{2-d}\,dz=\(1+\tf\,|y|^2\)^{2-\frac{d+n}2}\int_{\R^{d-n}}\(1+\tf\,|\zeta|^2\)^{2-d}\,d\zeta\,.
\]
Let $d\gamma(y):=(2\,\pi)^{-n/2}\,e^{-\frac12\,|y|^2}\,dy$ be the centred normalized Gaussian probability measure. We recall that $|\S^{k-1}|=2\,\pi^{k/2}/\Gamma(k/2)$ for any $k\in\N\setminus\{0\}$ and
\[
\int_0^{+\infty}r^{a-1}\(1+\tf\,r^2\)^{-b}\,dr=d^\frac a2\,\frac{\Gamma\(\frac a2\)\,\Gamma\(b-\frac a2\)}{2\,\Gamma(b)}
\]
if $0<a<2\,b$. Applying these formulas with $a=k=d-n$ and $b=d-2>2-n$, we find that
\[
\int_{\R^{d-n}}\(1+\tf\,|\zeta|^2\)^{2-d}\,d\zeta=(d\,\pi)^\frac{d-n}2\,\frac{\Gamma\(\frac{d+n}2-2\)}{\Gamma(d-2)}\,.
\]
Applying these formulas with $a=b=k=d\ge2$, we find that
\[
\ird{\(1+\tf\,|x|^2\)^{-d}}=c_d
\]
and, as a consequence
\[
\lim_{d\to+\infty}\frac1{c_d}\int_{\R^{d-n}}\(1+\tf\,|\zeta|^2\)^{2-d}\,d\zeta=\frac 4{(2\,\pi)^{n/2}}
\]
using Stirling's formula. Since
\[
\lim_{d\to+\infty}\(1+\tf\,|y|^2\)^{2-\frac{d+n}2}=e^{-\frac12\,|y|^2}\,,
\]
we obtain
\[
\lim_{d\to+\infty}\frac1{c_d}\ird{|\nabla g_d(y)|^2\(1+\tf\,|x|^2\)^{2-d}}=4\irng{|\nabla v|^2}\,.
\]
Similar computations show that
\begin{multline*}
\ird{|g_d(y)|^2\(1+\tf\,|x|^2\)^{-d}}=\irn{|g_d(y)|^2\(1+\tf\,|y|^2\)^{-\frac{d+n}2}}\int_{\R^{d-n}}\(1+\tf\,|\zeta|^2\)^{-d}\,d\zeta\,,\\
\lim_{d\to+\infty}\frac1{c_d}\ird{|g_d(y)|^2\(1+\tf\,|x|^2\)^{-d}}=\irng{|v|^2}\,,
\end{multline*}
and
\begin{multline*}
\ird{|g_d(y)|^p\(1+\tf\,|x|^2\)^{-d}}=\irn{|g_d(y)|^p\(1+\tf\,|y|^2\)^{-\frac{d+n}2}}\int_{\R^{d-n}}\(1+\tf\,|\zeta|^2\)^{-d}\,d\zeta\,,\\
\lim_{d\to+\infty}\frac{\mathcal C_{d,p}\,d^{d\frac{p-2}{2\,p}}}{4\,d\,c_d}\(\ird{|g_d(y)|^p\(1+\tf\,|x|^2\)^{-d}}\)^\frac2p=\(\irng{|v|^p}\)^\frac2p\,.
\end{multline*}
This completes the proof of Lemma~\ref{Lem:WGNS-PLS}.\end{proof}

\begin{proof}[Proof of Theorem~\ref{Thm:GNStoPLS}]

Applied to the function $u_d$, Inequality~\eqref{Ineq:GNS} is transformed into Inequality~\eqref{Ineq:wGNS} applied to
\[
g_d(x)=u_d\(\frac{2\,y}{1+\tf\,|x|^2}\)\quad\forall\,x=(y,z)\in\R^n\times\R^{d-n}\,.
\]
Notice that the factor $2^{\delta(p)/2}$, which arises from the stereographic projection and appears in~\eqref{up}, plays a role in the computation of the constant $\mathcal C_{d,p}$ and is taken into account in the limit as $d\to+\infty$. Since the right-hand side uniformly converges to $v(y)$ for any smooth and compactly supported function $v$, the same conclusion holds for Theorem~\ref{Thm:GNStoPLS} as for Lemma~\ref{Lem:WGNS-PLS}.
\end{proof}

\medskip It is a natural question to ask what happens in~\eqref{Ineq:GNS} to the marginals depending only on a finite number $n$ of variables if $p=2$ or in the case $2<p\le2^*=2\,d/(d-2)$. We may notice that $\lim_{d\to+\infty}2\,d/(d-2)=2$ and it is known, for instance from~\cite{MR1164616}, that one recovers the \emph{Gaussian logarithmic Sobolev inequality} as a limit case of Sobolev's inequality on $\S^d$ corresponding to $p=2^*$ when $d\to+\infty$. This is also true if we consider a sequence $(p_d)_{d\in\N}$ with $1<p_d<2^*$, depending on $d$, if its limit is also $2$, as shown next. By convention, when $p_d=2$, we consider the Gaussian logarithmic Sobolev inequality instead of~\eqref{Ineq:PLS}.
%-------------------------------------------------------------------------
\begin{proposition}\label{Prop:GNStoLSI} Let $n$ be a positive integer and consider a function $v\in\mathrm H^1(\R^n,dx)$ with compact support. For any $d\ge n$, large enough, let $u_d(\omega)=v\big(\omega_1/\sqrt d,\omega_2/\sqrt d,\ldots,\omega_n/\sqrt d\big)$ where $\omega=(\omega_1,\omega_2,\ldots,\omega_d)\in\R^{d+1}\supset\S^d$ is such that $|\omega|=1$, as in Theorem~\ref{Thm:GNStoPLS}. Then we have
\begin{multline*}
\lim_{d\to+\infty}d\(\nrmsd{\nabla u_d}2^2-\frac d2\isd{|u_d|^2\,\log\(\frac{|u_d|^2}{\nrms{u_d}2^2}\)}\)\\
=\nrmng{\nabla v}2^2-\frac12\irng{|v|^2\,\log\(\frac{|v|^2}{\nrmng v2^2}\)}\,.
\end{multline*}
If $(p_d)_{d\ge3}$ is a sequence of real numbers such that $p_d\in(1,2)\cup(2,2^*)$ and $\lim_{d\to+\infty}p_d=2$, then
\begin{multline*}
\lim_{d\to+\infty}d\(\nrmsd{\nabla u_d}2^2-\frac d{2-p_d}\(\nrmsd{u_d}2^2-\nrmsd{u_d}{p_d}^2\)\)\\
=\nrmng{\nabla v}2^2-\frac12\irng{|v|^2\,\log\(\frac{|v|^2}{\nrmng v2^2}\)}\,.
\end{multline*}
\end{proposition}
%-------------------------------------------------------------------------
\begin{proof} The proof is an adaptation of the proof of Theorem~\ref{Thm:GNStoPLS} and, in the case $p_d\neq2$, relies on the standard observation that
\[
\lim_{p\to2}\frac{\nrmng vp^2-\nrmng v2^2}{p-2}=\frac12\irng{|v|^2\,\log\(\frac{|v|^2}{\nrmng v2^2}\)}\,.
\]
As this computation raises no special difficulty, details are omitted.
\end{proof}

%%%%%%%%%%%%%%%%%%%%%%%%%%%%%%%%%%%%%%%%%%%%%%%%%%%%%%%%%%%%%%%%%%%%%%%%%%
%%%%%%%%%%%%%%%%%%%%%%%%%%%%%%%%%%%%%%%%%%%%%%%%%%%%%%%%%%%%%%%%%%%%%%%%%%
\section{Entropy methods and nonlinear flows for Gaussian measures}\label{Sec:Nonlinear}

In this section, we prove the result of Theorem~\ref{Thm:BE-Gaussian} for the Gaussian measure $d\gamma$ and extend it to the slightly more general framework of a uniformly strictly log-concave measure $d\mu$, before drawing some consequences. Most of the results are similar to computations usually done on the sphere, but we are not aware of the use of nonlinear flows ($m\neq1$) in the context of Gaussian measures. This approach is very natural in the perspective of spheres in the large-dimensional limit.

%%%%%%%%%%%%%%%%%%%%%%%%%%%%%%%%%%%%%%%%%%%%%%%%%%%%%%%%%%%%%%%%%%%%%%%%%%
\subsection{Gaussian interpolation inequalities: a proof by the \emph{carr\'e du champ} method}\label{Sec:Improved}
On $\R^n$, let us consider the probability measure
\be{dmu}
d\mu=Z^{-1}\,e^{-\phi}\,dy\quad\mbox{with}\quad Z=\irn{e^{-\phi}}
\ee
and redefine the \emph{Ornstein-Uhlenbeck operator}
by
\be{genOU}
\mathcal L:=\Delta-\nabla\phi\cdot\nabla
\ee
on $\mathrm L^2(\R^n,d\mu)$. This generalizes the case of the harmonic potential $\phi(y)=\frac12\,|y|^2$ considered in the introduction. We assume that $\phi$ satisfies the \emph{Bakry-Emery condition}
\be{BLW1}
\mathrm{Hess}\,\phi\ge\lambda_\star\,\mathrm{Id}\quad\mbox{a.e.}
\ee
for some $\lambda_\star>0$. The harmonic potential corresponds to the equality case with $\lambda_\star=1$. Under Assumption~\eqref{BLW1}, it is well known (see for instance~\cite[Section~7.6.2]{MR3155209}) that, with $\lambda=\lambda_\star$ and for any $p\in[1,2)$,
\be{Interpolation}
\nrmnmu{\nabla f}2^2\ge\frac\lambda{2-p}\(\nrmnmu f2^2-\nrmnmu fp^2\)\quad\forall\,f\in\mathrm H^1(\R^n,d\mu)
\ee
and also, by taking the limit as $p\to2_-$, that
\[
\nrmnmu{\nabla f}2^2\ge\frac\lambda2\irnmu{|f|^2\,\log\(\frac{|f|^2}{\nrmnmu f2^2}\)}\quad\forall\,f\in\mathrm H^1(\R^n,d\mu)\,.
\]
The classical proof by the carr\'e du champ, as in~\cite{MR772092,MR889476,MR1842428}, relies on the \emph{Ornstein-Uhlenbeck flow} $\partial_t\rho=\mathcal L\rho$ applied to the solution with initial datum $\rho(t=0,\cdot)=|f|^p$. Here we consider a more general strategy and compute as in the carr\'e du champ method using the \emph{nonlinear diffusion flow}
\be{rhoFDE}
\frac{\partial\rho}{\partial t}=\frac1m\,\mathcal L\rho^m\quad\forall\,t\ge0\,,\quad\rho(t=0,\cdot)=|f|^p\,.
\ee
There is no difficulty in proving global existence of a positive bounded solution. The case of a non-negative solution is also covered by a standard approximation of the initial datum. In the case of the harmonic potential, the key of the existence proof is to combine an $\mathrm L^1$-contraction property and the Maximum Principle as in the classical theory of nonlinear diffusions: see for instance~\cite[Chapter 9]{zbMATH05071489} for a general reference on this topic. In the case of a general potential $\phi$, for existence and uniqueness we refer to~\cite[Section~3]{MR1853037} and to~\cite{MR3497125} for a general overview of the topic. Concerning the $t\to+\infty$ limit, it is enough to consider initial data bounded from above and from below by two positive constants: this property is preserved along the flow and the nonlinear diffusion term can then be estimated using linear operators, so that the solution converges to a unique constant determined by mass conservation. The proofs can be adapted from~\cite{MR2481073,MR1853037} and from the results on the sphere on the other hand: see~\cite{BDS2022-Sphere} and references therein. As they present no essential difficulty, they will be omitted. Our goal here is to understand the range of $m$ for which we have a monotonicity property of the deficit as in the case $m=1$. Let $m_\pm(p)$ be defined as in~\eqref{mpm}.
%-------------------------------------------------------------------------
\begin{theorem}\label{Thm:BE-logConcave} Assume that $n\ge1$, $p\in[1,2)$ and $m\in[m_-(p),m_+(p)]$. We consider the measure $d\mu$ as in~\eqref{dmu} such that~\eqref{BLW1} holds for some $\lambda_\star>0$. If $\rho>0$ solves~\eqref{rhoFDE} with $\mathcal L$ defined by~\eqref{genOU} for an initial datum $\rho(t=0,\cdot)=|f|^p$ in $\mathrm L^{2/p}\cap\mathrm L^1(\R^n,d\mu)$, then
\[
\frac d{dt}\(\nrmnmu{\nabla\rho^{1/p}}2^2-\frac{\lambda_\star}{2-p}\(\nrmnmu{\rho^{1/p}}2^2-\nrmnmu\rho1^{2/p}\)\)\le0\quad\forall\,t>0\,.
\]
\end{theorem}
%-------------------------------------------------------------------------
\noindent The result of Theorem~\ref{Thm:BE-logConcave} is new for $m\neq1$. Theorem~\ref{Thm:BE-Gaussian} corresponds to the special case of the harmonic potential $\phi(y)=\frac12\,|y|^2$ with $v=\rho^{1/p}$ in Theorem~\ref{Thm:BE-logConcave}. The range of the admissible parameters $(m,p)$ is the same in Theorems~\ref{Thm:BE-Gaussian} and~\ref{Thm:BE-logConcave}, and shown in Fig.~\ref{F2}. With the additional observation that $\rho(t,\cdot)$ converges to a constant as $t\to+\infty$ so that the limit of the deficit is $0$, the monotonicity of the deficit of Theorem~\ref{Thm:BE-logConcave} provides us with a proof of~\eqref{Interpolation}. Indeed, as a monotone non-increasing function with limit $0$ as $t\to+\infty$, the deficit is non-negative for any $t\ge0$ and, as a special case, for the deficit written for the initial datum. This, written for $\rho(t=0,\cdot)=|f|^p$, is precisely Inequality~\eqref{Interpolation}.

\begin{proof}[Proof of Theorem~\ref{Thm:BE-logConcave}] In order to do computations, a very convenient reformulation is obtained with the substitution
\[
w^{\beta\,p}=\rho\,.
\]
Then, $w$ is a solution of the PDE 
\be{flow}
\frac{\partial w}{\partial t}=w^{2-2\beta}\(\mathcal Lw+\kappa\,\frac{|\nabla w|^2}w\)
\ee
for any $t\ge0$, with initial datum $w(t=0,\cdot)=|f|^{1/\beta}$, where
\be{kappa-beta}
\kappa:=\beta\,(p-2)+1\quad\mbox{and}\quad\beta:=\frac2{2-p\,(1-m)}\,.
\ee
A first computation shows that $\irnmu{w^{\beta p}}=\irnmu{|f|^p}$ is independent of $t$ because
\[
\frac d{dt}\irnmu{w^{\beta p}}=\beta\,p\irnmu{w^\kappa\(\mathcal Lw+\kappa\,\frac{|\nabla w|^2}w\)}=0\,.
\]
A second useful computation goes as follows:
\begin{align*}
-\frac1{2\,\beta^2}\,&\,\frac d{dt}\irnmu{\(\big|\nabla w^\beta\big|^2+\frac{\lambda_\star}{p-2}\,w^{2\beta}\)}\\
=&\irnmu{\(\mathcal Lw+(\beta-1)\,\frac{|\nabla w|^2}w-\frac{\lambda_\star\,w}{\beta\,(p-2)}\)\(\mathcal Lw+\kappa\,\frac{|\nabla w|^2}w\)}\\
=&\irnmu{(\mathcal Lw)^2}+(\kappa+\beta-1)\irnmu{(\mathcal Lw)\,\frac{|\nabla w|^2}w}+\kappa\,(\beta-1)\irnmu{\frac{|\nabla w|^4}{w^2}}-\lambda_\star\irnmu{|\nabla w|^2}\,.
\end{align*}
The purely technical purpose of introducing the exponent $\beta$ is to make the last line of the above computation $2$-homogeneous in $w$, which makes the discussion easier to read. Inserting the two following estimates,
\begin{align*}
\irnmu{(\mathcal Lw)^2}=-\irnmu{\nabla w\cdot\nabla(\mathcal Lw)}&=-\irnmu{\nabla w\cdot(\mathcal L\nabla w)}+\irnmu{\nabla w\cdot[\mathcal L,\nabla]\,w}\\
&=\irnmu{\|\mathrm{Hess}w\|^2}+\irnmu{\nabla w\cdot[\mathcal L,\nabla]\,w}\\
&=\irnmu{\|\mathrm{Hess}w\|^2}+\irnmu{\mathrm{Hess}\,\phi:\nabla w\otimes\nabla w}\\
&\ge\irnmu{\|\mathrm{Hess}w\|^2}+\lambda_\star\irnmu{|\nabla w|^2}
\end{align*}
and
\[
\irnmu{(\mathcal Lw)\,\frac{|\nabla w|^2}w}=-\,2\irnmu{\mathrm{Hess}w:\frac{\nabla w\otimes\nabla w}w}+\irnmu{\frac{|\nabla w|^4}{w^2}}\,,
\]
we obtain that
\[
\frac d{dt}\irnmu{\(\big|\nabla w^\beta\big|^2+\frac{\lambda_\star}{p-2}\,w^{2\beta}\)}\le0
\]
if, for any function $w$, we have
\begin{multline*}
\mathcal Q_\beta[w]:=\irnmu{\|\mathrm{Hess}w\|^2}-2\,(\kappa+\beta-1)\irnmu{\mathrm{Hess}w:\frac{\nabla w\otimes\nabla w}w}\\
+\big(\kappa\,(\beta-1)+\kappa+\beta-1\big)\irnmu{\frac{|\nabla w|^4}{w^2}}\ge0\,.
\end{multline*}
A sufficient condition is obtained if the reduced discriminant is negative, that is, if
\[
(\kappa+\beta-1)^2-\big(\kappa\,(\beta-1)+\kappa+\beta-1\big)\le0\,.
\]
Altogether, this gives the condition
\be{Cdt-beta}
\beta_-(p)\le\beta\le\beta_+(p)\quad\mbox{with}\quad\beta_\pm(p):=\frac{1\pm\sqrt{(p-1)\,(2-p)}}{1-(p-1)\,(2-p)}\,.
\ee
See Fig.~\ref{F3}. Equivalently, written in terms of $m$, the condition is $m_-(p)\le m\le m_+(p)$ with $m_\pm(p)$ defined by~\eqref{mpm}. Summarizing our computations, we learn that
\be{ResQbeta}
\frac d{dt}\(\nrmnmu{\nabla w}2^2-\frac{\lambda_\star}{p-2}\(\nrmnmu wp^2-\nrmnmu w2^2\)\)=-\,2\,\beta^2\,\mathcal Q_\beta[w]\le0
\ee
because $\mathcal Q_\beta[w]$ is non-negative under Condition~\eqref{Cdt-beta}.

To make these computations rigorous, one has to justify all integrations by parts. This is by now rather standard and can be done using the following scheme.
\begin{enumerate}
\item Let $h>0$ be large enough and consider $\Omega_h:=\left\{x\in\R^n\,:\,\phi(x)<h\right\}$. We can consider the evolution equation restricted to $\Omega_h$, with no flux boundary conditions. Then apply the \emph{carr\'e du champ} method and keep track of the boundary terms. Since $\Omega_h$ is a bounded convex domain, these terms have a sign due to Grisvard's lemma: see for instance~\cite[Lemma 5.2]{MR2533926},~\cite[Proposition~4.2]{MR3150642},~\cite{MR3396210,MR2435196} and~\cite[Lemma~A.3]{MR3497125}.
\item Extend Inequality~\eqref{Interpolation} written on $\Omega_h$ to $\R^d$ by taking the limit as $h\to+\infty$ and then argue by density.
\end{enumerate}
This completes the proof of Theorem~\ref{Thm:BE-logConcave}.
\end{proof}
%-------------------------------------------------------------------------
\begin{figure}[ht]
\includegraphics[width=6cm]{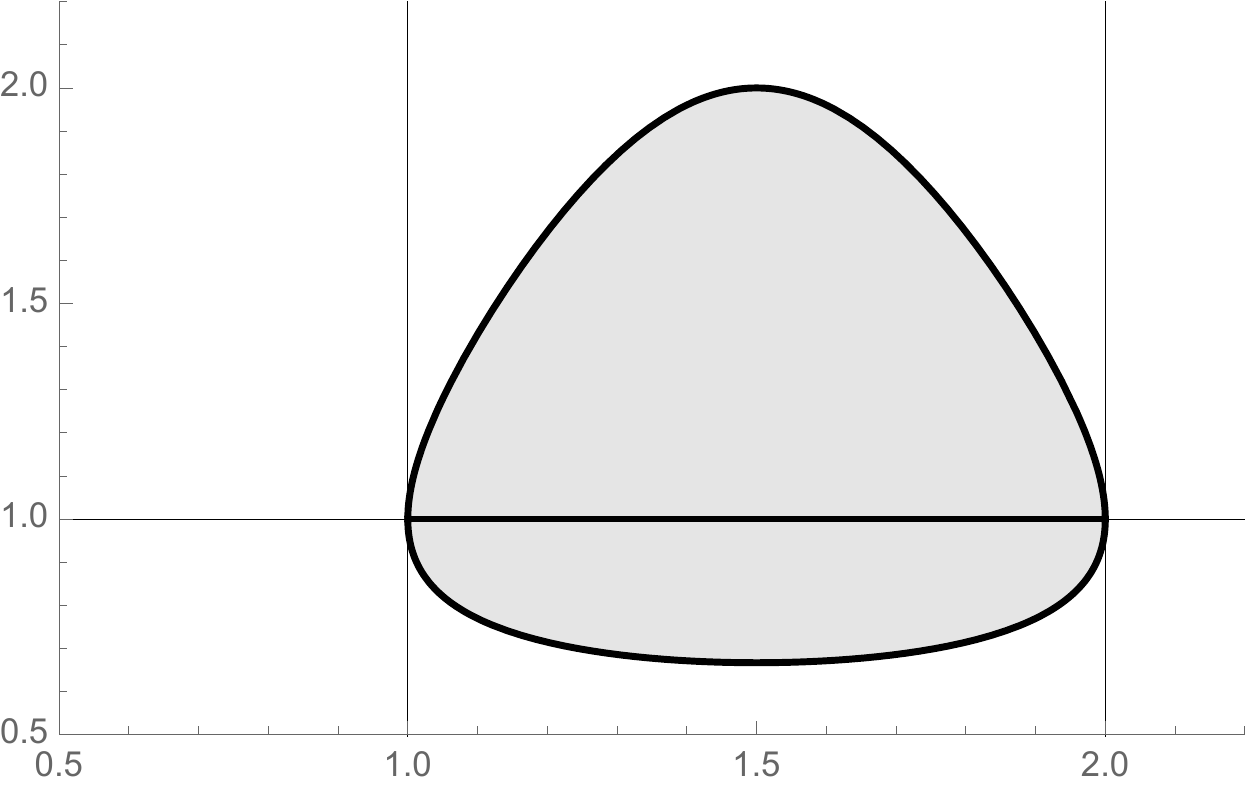}
\caption{\label{F3} The admissible parameters $p$ and $\beta$ correspond to the grey area and are independent of the dimension $n$. The boundary of the admissible set is tangent to the vertical lines $p=1$ at $(\beta,p)=(1,1)$ and $p=2$ at $(\beta,p)=(1,2)$. This figure corresponds to Fig.~\ref{F2} up to the transformation of $m\mapsto\beta=2/\big(2-p\,(1-m)\big)$ according to~\eqref{kappa-beta}.}
\end{figure}
%-------------------------------------------------------------------------
\begin{remark} For any $p\in[1,2)$, if the condition $m_-(p)\le m\le m_+(p)$ is not satisfied, then one can find a positive initial datum such that the function $v=w^\beta$ where $w$ solves~\eqref{flow} with $m$ and $\beta$ related by~\eqref{kappa-beta} is such that
\[
\frac d{dt}\(\nrmnmu{\nabla v}2^2-\frac d{p-2}\(\nrmnmu vp^2-\nrmnmu v2^2\)\)_{|t=0}>0\,.
\]
See~\cite{MR3640894} for a similar statement on the sphere and its proof. \end{remark}
%-------------------------------------------------------------------------

%%%%%%%%%%%%%%%%%%%%%%%%%%%%%%%%%%%%%%%%%%%%%%%%%%%%%%%%%%%%%%%%%%%%%%%%%%
\subsection{Improved inequalities based on the \emph{carr\'e du champ} method}\label{Sec:ImprovedBE}

In the proof of Theorem~\ref{Thm:BE-logConcave}, using only $\mathcal Q_\beta[w]\ge0$ is a crude estimate. Let us explain how one can obtain improved estimates by making a better use of $\mathcal Q_\beta[w]\ge0$. Under Condition~\eqref{Cdt-beta}, we can indeed rewrite $\mathcal Q_\beta[w]$ as an integral of a sum of squares,
\be{Qbeta}
\mathcal Q_\beta[w]=\irnmu{\left\|\mathrm{Hess}w-(\kappa+\beta-1)\,\frac{\nabla w\otimes\nabla w}w\right\|^2}+\delta\irnmu{\frac{|\nabla w|^4}{w^2}}
\ee
with
\be{delta}
\delta:=\kappa\,(\beta-1)+\kappa+\beta-1-(\kappa+\beta-1)^2=\big(\beta-\beta_-(p)\big)\,\big(\beta_+(p)-\beta\big)>0\,.
\ee
As in~\cite[Theorem~2]{MR2152502}, let us consider the special case $m=\beta=1$ of the linear flow. Let us define the \emph{entropy} and the \emph{Fisher information} by
\[
\mathcal E[w]:=\frac1{2-p}\(\nrmnmu w2^2-\nrmnmu wp^2\)\quad\mbox{and}\quad\mathcal I[w]:=\nrmnmu{\nabla w}2^2\,.
\]
Inequality~\eqref{Interpolation} amounts simply to
\[
\mathcal I[w]-\lambda_\star\,\mathcal E[w]\ge0\,.
\]
We can now state a first \emph{improved entropy--entropy production inequality}.
%-------------------------------------------------------------------------
\begin{proposition}\label{Prop:AD2005} Let $n$ be any positive integer. We consider the measure $d\mu$ as in~\eqref{dmu} such that~\eqref{BLW1} holds for some $\lambda_\star>0$. For any $p\in(1,2)$, let $\varphi(s):=1+s-(1+s)^{p-1}$. With the above notation, we have
\be{Ineq:AD}
\mathcal I[f]\ge\frac{\lambda_\star}{(2-p)^2}\,\nrmnmu fp^2\,\varphi\(\frac{(2-p)\,\mathcal E[f]}{\nrmnmu fp^2}\)\quad\forall\,f\in\mathrm H^1(\R^n,d\mu)\,.
\ee
\end{proposition}
%-------------------------------------------------------------------------
\noindent Since $p\in(1,2)$, the function $\varphi$ is convex with $\varphi(0)=0$ and $\varphi'(0)=2-p$, which implies in particular that $\varphi(s)\ge(2-p)\,s$ for any $s\ge0$, so that Inequality~\eqref{Ineq:AD} is stronger than~\eqref{Interpolation}. Inequality~\eqref{Ineq:AD} amounts~to
\[
\nrmnmu{\nabla f}2^2\ge\frac{\lambda_\star}{(2-p)^2}\(\nrmnmu f2^2-\nrmnmu f2^{2\,(p-1)}\,\nrmnmu fp^{2\,(2-p)}\)
\]
and can also be rewritten as
\begin{multline}\label{Interpolation2}
\nrmnmu{\nabla f}2^2-\frac{\lambda_\star}{2-p}\(\nrmnmu f2^2-\nrmnmu fp^2\)\\
\ge\frac{\lambda_\star}{(2-p)^2}\((p-1)\,\nrmnmu f2^2+(2-p)\,\nrmnmu fp^2-\nrmnmu f2^{2\,(p-1)}\,\nrmnmu fp^{2\,(2-p)}\)\,.
\end{multline}
We claim no originality in either Proposition~\ref{Prop:AD2005} or Inequality~\eqref{Interpolation2} and refer to~\cite[Theorem~2]{MR2152502} and~\cite[Ineq.~(3.3)]{MR2375056} for earlier results. Also see~\cite[Theorem~2.1]{MR4095470} for an improved inequality like~\eqref{Interpolation2} in the case of the sphere. Let us give the main ideas of the proof for completeness. We refer to~\cite[Appendix~B.4]{BDS2022-Sphere} for a fully detailed estimates in the similar case of the sphere.
\begin{proof}[Sketch of a proof] With $\beta=1$, notice that~\eqref{Qbeta} holds with $\delta=(2-p)\,(p-1)$. Using the Cauchy-Schwarz inequality, we obtain
\[
\big(\mathcal I[w]\big)^2=\(\irnmu{|\nabla w|^2}\)^2\le\irnmu{|w|^2}\irnmu{\frac{|\nabla w|^4}{w^2}}
\]
and, with $M:=\(\irnmu{|w|^p}\)^{2/p}$, we can also write that
\[
\irnmu{|w|^2}=(2-p)\,\mathcal E[w]+M\,.
\]
Altogether, with $\mathsf e(t):=(2-p)\,M^{-1}\,\mathcal E[w(t,\cdot)]$, if $w$ solves~\eqref{flow} with $\beta=1$, then
\be{e'}
\mathsf e'=-\,2\,\beta^2\,(2-p)\irnmu{|\nabla w|^2}
\ee
and we have the differential inequality
\[
\mathsf e''+2\,\lambda_\star\,\mathsf e'-(p-1)\,\frac{(\mathsf e')^2}{1+\mathsf e}\ge0\,.
\]
We claim that $(2-p)\,\mathsf e'+2\,\lambda_\star\(1+\mathsf e-(1+\mathsf e)^{p-1}\)\le0$, which follows from the observation that the equation
\[
y''+\mathsf a\,y'-\mathsf b\,\frac{(y')^2}{1+y}=0
\]
can be solved using the ansatz $y'=\frac{\mathsf a}{\mathsf b-1}\,\varphi(y)$ with $\varphi(0)=0$ if
\be{EDO:b}
\varphi'-\mathsf b\,\frac\varphi{1+s}=1-\mathsf b\,.
\ee
It is straightforward to check that the unique solution is $\varphi(s)=1+s-(1+s)^{\mathsf b}$. With $\mathsf a=2\,\lambda_\star$ and $\mathsf b=p-1$, we obtain
\[
\(\mathsf e'+\frac{2\,\lambda_\star}{2-p}\,\varphi(\mathsf e)\)'\ge(p-1)\,\frac{\mathsf e'}{1+\mathsf e}\(\mathsf e'+\frac{2\,\lambda_\star}{2-p}\,\varphi(\mathsf e)\)\,.
\]
By integrating from $t=0$ to $+\infty$ using the facts that $\mathsf e'\le0$, $\mathsf e'+\frac{2\,\lambda_\star}{2-p}\,\mathsf e\le0$ and $\lim_{t\to+\infty}\mathsf e'(t)=\lim_{t\to+\infty}\mathsf e(t)=0$, we conclude that
\[
\frac{2-p}M\(\mathcal I[w]-\frac{\lambda_\star\,M}{(2-p)^2}\,\varphi\(\frac{(2-p)\,\mathcal E[w]}M\)\)=-\,\frac12\(\mathsf e'+\frac{2\,\lambda_\star}{2-p}\,\varphi(\mathsf e)\)\ge0\quad\forall\,t\ge0
\]
\end{proof}

Inspired once more by the results on the sphere (see~\cite{MR4095470} and earlier references therein), it is a very natural question to wonder if \emph{improved entropy--entropy production inequalities} can be achieved with $\beta\neq1$. The answer goes as follows. Let us consider the function $\varphi_\beta$ given by
\[
\chi_\beta(s)=\big(1-\mathsf b(\beta)\big)\(1+s-(1+s)^{\mathsf b(\beta)}\)\quad\mbox{where}\quad\mathsf b(\beta)=\frac{\delta(\beta)}{\beta^2}\,\frac{2-p}{\lambda_\star
}
\]
and $\delta(\beta)$ is defined by~\eqref{delta}.
%-------------------------------------------------------------------------
\begin{proposition}\label{Prop:Qbeta} Let $n$ be any positive integer. We consider the measure $d\mu$ as in~\eqref{dmu} such that~\eqref{BLW1} holds for some $\lambda_\star>0$. For any $p\in(1,2)$, take $\beta\in\big[1,\beta_+(p)\big]$. With the same notation as in Proposition~\ref{Prop:AD2005} and $\chi_\beta$ as above, we have
\[
\mathcal I[f] \ge \nrmnmu fp^2\,\chi_\beta\(\frac{(2-p)\,\mathcal E[f]}{\nrmnmu fp^2}\)\quad\forall\,f\in\mathrm H^1(\R^n,d\mu)\,.
\]
\end{proposition}
%-------------------------------------------------------------------------
\begin{proof} With the notation
\[
\mathsf e:=\frac{\Nrmnmu{w^\beta}2^2}{\Nrmnmu{w^\beta}p^2}\quad\mbox{and}\quad\mathsf i:=\frac{\Nrmnmu{\nabla\big(w^\beta\big)}2^2}{\Nrmnmu{w^\beta}p^2}\,,
\]
we have now to consider the differential inequality
\[
\(\mathsf i'-\frac{\lambda_\star}{2-p}\,\mathsf e\)'\le-\,\delta\,\beta^2\irnmu{\frac{|\nabla w|^4}{w^2}}\,,
\]
which directly comes from the \textit{carr\'e du champ} computation~\eqref{ResQbeta} with $\mathcal{Q}_\beta$ and $\delta$ defined respectively in~\eqref{Qbeta} and~\eqref{delta}. The key ingredient is to replace an estimate due to Demange in the case of the sphere for $p>2$ (see~\cite{Demange-PhD,MR2381156} and also ~[Lemma~15,~(iii)]) by its counterpart for log-concave measures and $p\in(1,2)$. Compared to the linear case, the Cauchy-Schwarz inequality in the proof of Proposition~\ref{Prop:AD2005} has to be replaced by the H\"older inequalities
\begin{align*}
&\irnmu{|\nabla w|^2}=\irnmu{\(\frac{|\nabla w|^2}w\,w\,1\)}\le\(\irnmu{\frac{|\nabla w|^4}{w^2}}\)^\frac12\(\irnmu{|w|^{2\,\beta}}\)^\frac1{2\,\beta}1^\frac{\beta-1}{2\,\beta}\,,\\
&\frac1{\beta^2}\irnmu{\big|\nabla w^\beta\big|^2}=\irnmu{\(\frac{|\nabla w|^2}w\,w^{2\,\beta-1}\,1\)}\le\(\irnmu{\frac{|\nabla w|^4}{w^2}}\)^\frac12\(\irnmu{|w|^{2\,\beta}}\)^\frac{2\,\beta-1}{2\,\beta}1^\frac1{2\,\beta}\,,
\end{align*}
after observing that $\beta\ge\beta_-(p)\ge\beta_-(3/2)=2/3>1/2$, so that
\[
\irnmu{\frac{|\nabla w|^4}{w^2}}\ge\frac{\irnmu{|\nabla w|^2}\,\irnmu{\big|\nabla w^\beta\big|^2}}{\beta^2\irnmu{w^{2\beta}}}=-\,\frac{\mathsf i\,\mathsf e'}{\beta^4\,(1+\mathsf e)}\,.
\]
where the last equality is a consequence of the definition of $\mathsf i$ and~\eqref{e'}. Hence,
\[
\(\mathsf i-\frac{\lambda_\star}{2-p}\,\mathsf e\)'\le\frac\delta{\beta^2}\,\frac{\mathsf i\,\mathsf e'}{1+\mathsf e}\,.
\]
Let us compute
\[
\(\mathsf i-\frac{\lambda_\star}{2-p}\,\chi(\mathsf e)\)'=\(\mathsf i-\frac{\lambda_\star}{2-p}\,\mathsf e\)'+\frac{\lambda_\star}{2-p}\(\mathsf e-\chi(\mathsf e)\)'\le\frac\delta{\beta^2}\,\frac{\mathsf e'}{1+\mathsf e}\(\mathsf i-\frac{\lambda_\star}{2-p}\,\chi(\mathsf e)\)
\]
on the condition that $\chi$ solves
\[
\chi'(s)=1+\frac\delta{\beta^2}\,\frac{2-p}{\lambda_\star}\,\frac{\chi(s)}{1+s}\,,\quad\chi(0)=0\,.
\]
The solution $\chi(s)=(1-\mathsf b)\,\varphi(s)$ is such that $\varphi$ solves~\eqref{EDO:b} with $\mathsf b=\delta\,\beta^{-2}\,(2-p)/\lambda_\star$, which shows that $\chi=\chi_\beta$. The proof then follows from the same considerations as for Proposition~\ref{Prop:AD2005} (also see~\cite[Appendix~B.4]{BDS2022-Sphere} for details).\end{proof}

If $\phi(y)=|y|^2/2$ is the harmonic potential so that $d\mu=d\gamma$, by testing~\eqref{Ineq:PLS} with $f_\varepsilon(y)=1+\varepsilon\,y_1$ where $y_1$ denotes the first coordinate of $y=(y_1,y_2,\ldots,y_n)\in\R^n$, we find that
\[
\nrmng{\nabla f_\varepsilon}2^2-\frac1{2-p}\(\nrmng{f_\varepsilon}2^2-\nrmng{f_\varepsilon}p^2\)=\frac12\,(p-1)\,\varepsilon^4+O(\varepsilon^5)\quad\mbox{as}\quad\varepsilon\to0\,.
\]
This is the standard computation for checking that $\lambda=\lambda_\star=1$ is the optimal constant in~\eqref{Ineq:PLS}. Since
\[
\mathcal I[f_\varepsilon]-\frac1{(2-p)^2}\,\nrmng{f_\varepsilon}p^2\,\varphi\(\frac{(2-p)\,\mathcal E[f_\varepsilon]}{\nrmng{f_\varepsilon}p^2}\)=\frac12\,(p-1)^2\,\varepsilon^4+O(\varepsilon^5)\quad\mbox{as}\quad\varepsilon\to0\,,
\]
we also learn that~\eqref{Ineq:AD} involves the optimal exponent at least in the limit as $\varepsilon\to0$. After observing that $\nrmng{\nabla f_\varepsilon}2^2=\varepsilon^2$, we may wonder whether the \emph{deficit}
\[
\nrmng{\nabla f}2^2-\frac1{2-p}\(\nrmng f2^2-\nrmng fp^2\)
\]
measures the distance in terms of $\nrmng{\nabla f}2^4$. The detailed answer is not limited to the case $\mu=\gamma$ and goes as follows. For simplicity, we take $\beta=1$ and consider $\varphi(s):=1+s-(1+s)^{p-1}$ as in Proposition~\ref{Prop:AD2005}. We recall that $\varphi$ is monotone increasing and convex on $\R^+$, such that $\varphi'(0)=2-p$, hence invertible of inverse $\varphi^{-1}$ such that $\psi(t):=t-(2-p)\,\varphi^{-1}(t)$ is also a convex, non-negative, monotone increasing function.
%-------------------------------------------------------------------------
\begin{corollary}\label{Cor:4thOrder} Let $p\in(1,2)$ and $n$ be a positive integer. We consider the measure $d\mu$ as in~\eqref{dmu} such that~\eqref{BLW1} holds for some $\lambda_\star>0$. With $\psi$ as above, for any $f\in\mathrm H^1(d\gamma)$ we have
\[
\nrmnmu{\nabla f}2^2-\frac{\lambda_\star}{2-p}\(\nrmnmu f2^2-\nrmnmu fp^2\)\ge\frac{\lambda_\star}{2-p}\,\nrmnmu fp^2\,\psi\(\frac{2-p}{\lambda_\star}\,\frac{\nrmnmu{\nabla f}2^2}{\nrmnmu fp^2}\)\,.
\]
Moreover, there is some $\kappa>0$ such that
\be{kappa}
\nrmnmu{\nabla f}2^2-\frac{\lambda_\star}{2-p}\(\nrmnmu f2^2-\nrmnmu fp^2\)\ge\frac{\kappa\,\nrmnmu{\nabla f}2^4}{\nrmnmu{\nabla f}2^2+\frac{\lambda_\star}{2-p}\,\nrmnmu f2^2}\,.
\ee
\end{corollary}
%-------------------------------------------------------------------------
\noindent The constant $\kappa$ depends only on $p$ and its value is estimated in the proof below.
\begin{proof} Let $M=\nrmnmu fp^2$. We deduce from Proposition~\ref{Prop:AD2005} that
\[
\mathsf i:=\frac{(2-p)^2}{\lambda_\star\,M}\,\mathcal I[f]\ge\varphi(\mathsf e)\quad\mbox{where}\quad\mathsf e:=\frac{(2-p)\,\mathcal E[f]}{\nrmnmu fp^2}\,,
\]
which is equivalent to $-\,\mathsf e\ge-\,\varphi^{-1}(\mathsf i)$, so that
\[
\mathsf i-(2-p)\,\mathsf e\ge\mathsf i-(2-p)\,\varphi^{-1}(\mathsf i)=\psi(\mathsf i)
\]
and, as a consequence,
\[
\mathcal I[f]-\lambda_\star\,\mathcal E[f]=\frac{\lambda_\star\,M}{(2-p)^2}\,\big(\mathsf i-(2-p)\,\mathsf e\big)\ge\frac{\lambda_\star\,M}{(2-p)^2}\,\psi(\mathsf i)=\frac{\lambda_\star\,M}{(2-p)^2}\,\psi\(\frac{2-p}{\lambda_\star\,M}\,\nrmnmu{\nabla f}2^2\)\,.
\]
Since $\varphi(s)\sim s$ as $s\to+\infty$, we deduce that $\psi(t)\sim(p-1)\,t$ as $t\to+\infty$. On the other hand, since
\[
\psi''(t)=(2-p)\,\frac{\varphi''(s)}{\big(\varphi'(s)\big)^3}\quad\mbox{with}\quad s=\varphi^{-1}(t)\,,
\]
we learn that $\psi''(0)=(p-1)/(2-p)>0$ and $t\mapsto\psi''(t)$ is non-increasing because $\varphi'(s)^5\,\psi'''(t)=\varphi'(s)\,\varphi'''(s)-\varphi''(s)^2<0$. This allows us to define
\[
\kappa:=\inf_{t>0}t^{-2}\,(1+t)\,\psi(t)\,.
\]
Using $\psi(0)=\psi'(0)=0$, we know that $\kappa>0$ and $\psi(t)\ge\kappa\,t^2/(1+t)$ concludes the proof using $\nrmnmu fp\le\nrmnmu f2$.
\end{proof}
%-------------------------------------------------------------------------
\begin{remark} If $\phi(y)=|y|^2/2$ is the harmonic potential so that $d\mu=d\gamma$, then $\lambda=\lambda_\star=1$ is the optimal constant in~\eqref{Ineq:PLS} and the results of Proposition~\ref{Prop:AD2005} and Corollary~\ref{Cor:4thOrder} are both stability estimates. Even in the general case of a measure $d\mu$ as in~\eqref{dmu} such that~\eqref{BLW1} holds for some $\lambda_\star>0$, Proposition~\ref{Prop:AD2005} and Corollary~\ref{Cor:4thOrder} provide improvements to the basic inequality with a (generically non-optimal) constant $\lambda_\star$.
\end{remark}
%-------------------------------------------------------------------------

%%%%%%%%%%%%%%%%%%%%%%%%%%%%%%%%%%%%%%%%%%%%%%%%%%%%%%%%%%%%%%%%%%%%%%%%%%
\subsection{From the \emph{carr\'e du champ} method to Obata's theorem}

As a side result, we consider an improvement of the \emph{carr\'e du champ} method as in~\cite[Theorem~2.1]{MR1631581} or~\cite{MR3229793}, which goes as follows. Let us consider the optimal constant $\lambda_1>0$ in the Poincar\'e inequality
\be{Poincare:mu}
\irnmu{|\nabla w|^2}\ge\lambda_1\irnmu{|w-\bar w|^2}\quad\forall\,w\in\mathrm H^1(\R^n,d\mu)\,,
\ee
where $\bar w:=\irnmu w$. By expanding $\irnmu{\big(\mathcal Lw+\lambda_1\,(w-\bar w)\big)^2}\ge0$, we obtain
\be{lambda1}
\irnmu{(\mathcal Lw)^2}\ge\lambda_1\irnmu{|\nabla w|^2}\,.
\ee
On the other hand, by the computation of Section~\ref{Sec:Improved} with $\beta=p=1$, we know that
\[
-\frac12\,\frac d{dt}\irnmu{\(|\nabla w|^2+\frac{\lambda_\star}{p-2}\,w^2\)}=\irnmu{(\mathcal Lw)^2}-\lambda_\star\irnmu{|\nabla w|^2}=\mathcal Q_1[w]\ge0\,.
\]
which proves, for instance using~\eqref{ResQbeta}, that
\[
\lambda_1\ge\lambda_\star\,.
\]
This result goes back to the work of Obata and the \emph{carr\'e du champ} method is in fact very close to the spirit of the historical proof: see~\cite[p.~327]{MR303464}.
%-------------------------------------------------------------------------
\begin{lemma}\label{Lem:Veron-Licois} Assume that $n\ge1$, $p\in[1,2)$ and consider the measure $d\mu$ as in~\eqref{dmu} such that~\eqref{BLW1} holds for some $\lambda_\star>0$. Then~\eqref{Interpolation} holds with 
\be{lambda-lambdastar}
\lambda=(2-p)\,\lambda_1+(p-1)\,\lambda_\star\,.
\ee
As a consequence of~\eqref{lambda-lambdastar}, we have $\lambda\ge\lambda_\star$ with equality for the optimal value of $\lambda$ in~\eqref{Interpolation} if and only if $\lambda_\star=\lambda_1$ and $\phi(y)=\lambda_\star\,|y-y_0|^2/2$ for some $y_0\in\R^n$.
\end{lemma}
%-------------------------------------------------------------------------
\begin{proof} The \emph{carr\'e du champ} method applied as in Section~\ref{Sec:Improved} with $p=1$ shows that $\lambda_1\ge\lambda_\star$. Coming back to the computations of Section~\ref{Sec:Improved}, we can rearrange the integrals in the expression $\mathcal Q_\beta[w]$ differently and get
\begin{align*}
&\hspace*{-18pt}-\frac1{2\,\beta^2}\,\frac d{dt}\irnmu{\(\big|\nabla w^\beta\big|^2+\frac\lambda{p-2}\,w^{2\beta}\)}\\
&=(1-\theta)\irnmu{(\mathcal Lw)^2}-\lambda\irnmu{|\nabla w|^2}\\
&\hspace*{1cm}+\theta\irnmu{(\mathcal Lw)^2}+(\kappa+\beta-1)\irnmu{(\mathcal Lw)\,\frac{|\nabla w|^2}w}+\kappa\,(\beta-1)\irnmu{\frac{|\nabla w|^4}{w^2}}\\
&=\big((1-\theta)\,\lambda_1+\theta\,\lambda_\star-\lambda\big)\irnmu{|\nabla w|^2}\\
&\hspace*{1cm}+\theta\(\irnmu{\|\mathrm{Hess}w\|^2}+\lambda_\star\irnmu{|\nabla w|^2}\)\\
&\hspace*{1cm}-2\,(\kappa+\beta-1)\irnmu{\mathrm{Hess}w:\frac{\nabla w\otimes\nabla\,w}w}+\big(\kappa\,(\beta-1)+\kappa+\beta-1\big)\irnmu{\frac{|\nabla w|^4}{w^2}}
\end{align*}
where the $(1-\theta)\,\lambda_1$ term factor from~\eqref{lambda1}. With the choice of $\theta$ such that
\[
(\kappa+\beta-1)^2-\theta\,\big(\kappa\,(\beta-1)+\kappa+\beta-1\big)=0\,,
\]
which means $\theta=\theta(\beta)$ with
\[
\theta(\beta):=\frac{(p-1)^2\,\beta^2}{(p-2)\,\beta^2+2\,\beta-1}\,,
\]
we can write that
\begin{multline*}
\theta\irnmu{\|\mathrm{Hess}w\|^2}-2\,(\kappa+\beta-1)\irnmu{\mathrm{Hess}w:\frac{\nabla w\otimes\nabla\,w}w}+\big(\kappa\,(\beta-1)+\kappa+\beta-1\big)\irnmu{\frac{|\nabla w|^4}{w^2}}\\=
\theta\irnmu{\left\|\mathrm{Hess}w-\frac{\beta\,(p-1)}\theta\,\frac{\nabla w\otimes\nabla\,w}w\right\|^2}\ge0\,.
\end{multline*}
Altogether, we have shown that
\[
\frac d{dt}\irnmu{\(\big|\nabla w^\beta\big|^2+\frac\lambda{p-2}\,w^{2\beta}\)}\le0
\]
if $\lambda=(1-\theta)\,\lambda_1+\theta\,\lambda_\star$. Recall that $\nrmnmu{w^\beta}p$ is independent of $t$ so that the deficit functional associated to~\eqref{Interpolation} is monotone non-increasing. Also notice that $\theta\big(\beta_\pm(p)\big)=1$ for any $p\in[1,2)$. The observation that $\min_{\beta\in[\beta_-(p),\beta_+(p)]}\theta(\beta)=\theta(1)=p-1$ completes the proof of~\eqref{lambda-lambdastar}.

With $\beta=1$ and $p=1$, the computation in the proof of Theorem~\ref{Thm:BE-logConcave} shows that, if the initial datum $w(t=0,\cdot)$ is an optimal function for the Poincar\'e inequality~\eqref{Poincare:mu}, then
\[
0=\frac d{dt}\irnmu{\(|\nabla w|^2-\lambda_\star\,w^2\)}=-\,2\(\mathcal Q_1[w]+\irnmu{\(\mathrm{Hess}\,\phi-\lambda_\star\mathrm{Id}\):\nabla w\otimes\nabla w}\)
\]
at $t=0$. Here we keep all terms and in particular do not use the fact that $\mathrm{Hess}\,\phi-\lambda_\star\mathrm{Id}\ge0$ a.e.~in the sense of positive matrices. Since $\nabla w\neq0$ a.e.~and $\mathcal Q_1[w]\ge0$, we find that $\mathrm{Hess}\,\phi-\lambda_\star\mathrm{Id}=0$ a.e. This completes the proof in the equality case $\lambda_\star=\lambda_1$.\end{proof}

%-------------------------------------------------------------------------
\begin{remark} The proof of Lemma~\ref{Lem:Veron-Licois} is reminiscent of~\cite{MR1631581,MR3229793}. The result when $\lambda_\star=\lambda_1$ points in the direction of \emph{Obata's theorem} (also known as the \emph{Obata–Lichnerowicz theorem}) and in some sense, it is the analogue for Gaussian measures of the result of~\cite[p.~135]{MR0124009} (also see for instance~\cite[p.~179]{MR0282313}) on the sphere. The case $\lambda_\star=\lambda_1$ is easy to understand in dimension $d=1$: with $\beta=1$ and $p=1$, we apply the computation of the proof of Lemma~\ref{Lem:Veron-Licois} to a function $u$ in the eigenspace associated with $\lambda_1$ and obtain that $u''=0$ almost everywhere. This means that $u(y)=a\,y+b$ for some real constants $a\neq0$ and $b$, and there is no loss of generality if we take $a=1$. Using now the eigenvalue equation $\mathcal Lu+\lambda_1\,u=0$, we read that $\phi'(y)=\lambda_1\,(y-b)$, which means that $\phi$ is a harmonic potential. In higher dimensions, one has to remember that Inequality~\eqref{Interpolation} can be tensorized on product spaces: see for instance~\cite{MR1796718,MR2081075,MR3884261}. This is however responsible for some technicalities, which are dealt with in greatest generality, \emph{e.g.}, in~\cite{MR4150652}. \end{remark}
%-------------------------------------------------------------------------

%%%%%%%%%%%%%%%%%%%%%%%%%%%%%%%%%%%%%%%%%%%%%%%%%%%%%%%%%%%%%%%%%%%%%%%%%%
\subsection{Improved inequalities under orthogonality conditions}\label{Sec:ImprovedSpectral}

Let $\Pi_1$ be the $\mathrm L^2(\R^n,d\gamma)$orthogonal projection onto the space generated by the constants and the coordinate functions, corresponding to the Hermite polynomials of order less or equal than $1$. The following result was recently proved in~\cite[Appendix A]{BDS2022-Sphere} on the basis of Nelson's hypercontractivity estimate in~\cite[Theorem~3]{MR0343816} and its relation with Gross' logarithmic Sobolev inequality in~\cite{MR420249} (also see~\cite{MR1796718,MR2375056} for earlier results).
%-------------------------------------------------------------------------
\begin{proposition}\label{orth}
Let $n\ge1$ and $p\in[1,2)$. For any $f\in\mathrm H^1(\R^n,d\gamma)$, we have
\[
\nrmng{\nabla f}2^2-\frac 1{p-2}\(\nrmng fp^2-\nrmng f2^2\)\ge\frac12\,(2-p)\,\nrmng{\nabla(\mathrm{Id}-\Pi_1)f}2^2\,.
\]
\end{proposition}
%-------------------------------------------------------------------------
\noindent Compared to~\eqref{Ineq:PLS}, this result provides us with an \emph{improved entropy--entropy production inequalities} under orthogonality conditions. As noted in~\cite{BDS2022-Sphere}, such an improvement is not optimal. There are other possible approaches. For instance, a finer analysis of entropy methods has been used in~\cite{MR3640894} on the sphere, that could probably be adapted to the case of the Gaussian measure. Alternatively, one could use the convex interpolation of~\cite{MR1796718}, with the possible advantage that the result would not degenerate in the limit as $p\to2$ using the recent stability result of~\cite[Theorem~2]{DEFFL}.

%%%%%%%%%%%%%%%%%%%%%%%%%%%%%%%%%%%%%%%%%%%%%%%%%%%%%%%%%%%%%%%%%%%%%%%%%%
%%%%%%%%%%%%%%%%%%%%%%%%%%%%%%%%%%%%%%%%%%%%%%%%%%%%%%%%%%%%%%%%%%%%%%%%%%
\section{Stability results for the Gaussian measure in the subcritical range}\label{Sec:Stability}

The whole Section is devoted to the proof of Theorem~\ref{stability-subcritical-Gaussian}. We split it into four lemmas. The key estimate is obtained in Lemma~\ref{Lem:Stab4}. 

\vspace*{4pt}\noindent\textcircled1 Let us start with the easy case, far away from the optimizers of~\eqref{Ineq:PLS} in the sense that for some $\theta>0$, we assume
\be{Far}
\nrmng{\nabla f}2^2\ge\theta\,\nrmng f2^2\,.
\ee
By homogenity of the inequalities, we can fix $\nrmng f2^2=1$ without loss of generality.
%-------------------------------------------------------------------------
\begin{lemma}\label{Lem:Stab1} Let $n\ge1$ and $\theta\in(0,1)$. For any function $f\in\mathrm H^1(\R^n,d\gamma)$ such that $\nrmng f2=1$ and \hbox{$\nrmng{\nabla f}2^2\ge\theta$}, we have the estimate
\[
\nrmng{\nabla f}2^2-\frac1{2-p}\(\nrmng f2^2-\nrmng fp^2\)\ge\kappa_\star(\theta)\,\nrmng{\nabla f}2^2
\]
\end{lemma}
%-------------------------------------------------------------------------
\noindent In case~\eqref{Far}, this already proves the result of Theorem~\ref{stability-subcritical-Gaussian} with $c_{n,p}\le c_{n,p}^{(1)}:=\kappa_\star(\theta)$. The result of Lemma~\ref{Lem:Stab1} directly follows from~\eqref{kappa}, but we give a proof based on the convexity properties of the entropy, which directly explains where the constant comes from.
\begin{proof} From Corollary~\ref{Cor:4thOrder} with $\lambda_\star=1$ and H\"older's inequality, we obtain
\begin{multline*}
\nrmng{\nabla f}2^2-\frac1{2-p}\(\nrmng f2^2-\nrmng fp^2\)\\
\ge\frac{\nrmng fp^2}{2-p}\,\frac12\,\psi\((2-p)\,\theta\)\((2-p)\,\frac{\nrmng{\nabla f}2^2}{\nrmng fp^2}\)^2\\
\ge\frac12\,(2-p)\,\psi\((2-p)\,\theta\)\,\frac{\nrmng{\nabla f}2^4}{\nrmng fp^2}\ge\kappa_\star(\theta)\,\nrmng{\nabla f}2^2
\end{multline*}
with $\kappa_\star(\theta):=\frac12\,(2-p)\,\psi\((2-p)\,\theta\)\,\theta$ because $t\mapsto\psi''(t)$ is non-increasing, $\nrmng{\nabla f}2^2\ge\theta$ and $\nrmng fp^2\le\nrmng f2^2$.\end{proof}

\noindent\textcircled2 From now on we work in a neighbourhood of the constants which, by homogeneity of the inequalities, is defined as
\[
\nrmng{\nabla f}2^2\le\theta\,\nrmng f2^2\,.
\]
With $\theta>0$ small, we claim that $\irng f$ is close to $1$ if $\nrmng f2=1$.
%-------------------------------------------------------------------------
\begin{lemma}\label{Lem:Stab2} Let $n\ge1$ and $\theta\in(0,1)$. For any non-negative function
\be{Neighbourhood}
f\in\mathrm H^1(\R^n,d\gamma)\quad\mbox{such that}\quad\nrmng f2=1\quad\mbox{and}\quad\nrmng{\nabla f}2^2\le\theta\,,
\ee
we have the estimate
\[
\sqrt{1-\theta}\le\irng f\le1\,.
\]
\end{lemma}
%-------------------------------------------------------------------------
\begin{proof} With $\bar f:=\irng f$, the result follows from the Gaussian Poincar\'e inequality according to
\[
1=\nrmng f2^2=\irng{|f-\bar f|^2}+\bar f^2\le\nrmng{\nabla f}2^2+\bar f^2\le\theta+\bar f^2\,.
\]
\end{proof}

\noindent\textcircled{3} Assume that $f$ is as in~\eqref{Neighbourhood} and let us decompose $u_f(x):=f(x)/\irng f$ as
\[
u_f(x)=1+\varepsilon\,x\cdot\nu+\eta\,r(x)
\]
where $\nu\in\S^{n-1}$ is such that $\varepsilon\,\nu=\irng{x\,u_f(x)}$ with $\varepsilon>0$, $\eta$ is a positive number and $r$ is a function in $\mathrm H^1(\R^n,d\gamma)\cap(\mathrm{Id}-\Pi_1)\mathrm L^2(\R^n,d\gamma)$ such that $\nrmng{\nabla r}2=1$ and $\nrmng r2\le1/2$ by the Gaussian Poincar\'e inequality after taking into account the additional orthogonality condition $\irng{r\,x_i}=0$ for any $i=1$, $2$,\ldots,$n$.
%-------------------------------------------------------------------------
\begin{lemma}\label{Lem:Stab3} Let $n\ge1$ and $\theta\in(0,1)$. Let $f\in\mathrm H^1(\R^n,d\gamma)$ be such that~\eqref{Neighbourhood} holds. With the above notation, we have
\[
\nrmng{u_f}2^2=1+\varepsilon^2+\eta^2\,\nrmng r2^2\le1+\theta\quad\mbox{and}\quad\nrmng{\nabla u_f}2^2=\varepsilon^2+\eta^2\le\frac\theta{1-\theta}
\]
and, if $\,\eta>t\,\varepsilon^2$ for some $t>0$, then
\[
\nrmng{\nabla u_f}2^2-\frac1{2-p}\(\nrmng{u_f}2^2-\nrmng{u_f}p^2\)\ge\frac14\,(2-p)\(\eta^2+\frac{t^2\,\varepsilon^4}{1+\varepsilon^2+\eta^2}\)\,.
\]
\end{lemma}
%-------------------------------------------------------------------------
\noindent By homogeneity, if~\eqref{Neighbourhood} holds and $\eta>t\,\varepsilon^2$ for some given $t>0$, using Lemma~\ref{Lem:Stab2}, we obtain
\begin{multline*}
\nrmng{\nabla f}2^2-\frac1{2-p}\(\nrmng f2^2-\nrmng fp^2\)\\
\ge\frac14\,(2-p)\,(1-\theta)\(\nrmng{\nabla(\mathrm{Id}-\Pi_1)\,f}2^2+\frac{t^2\,\nrmng{\nabla(\Pi_1\,f)}2^4}{\nrmng{\nabla f}2^2+\nrmng f2^2}\)\,.
\end{multline*}
This case already covers the result of Theorem~\ref{stability-subcritical-Gaussian} with
\be{cnp1}
c_{n,p}\le c_{n,p}^{(1)}:=\frac14\,(2-p)\,(1-\theta)\min\big\{t,1\big\}\,.
\ee
\begin{proof} The result follows from Proposition~\ref{orth} and from the chain of elementary inequalities
\[
\eta^2\ge\frac12\(\eta^2+t^2\,\varepsilon^4\)\ge\frac12\(\eta^2+\frac{t^2\,\varepsilon^4}{1+\varepsilon^2+\eta^2}\)\,.
\]
\end{proof}

\noindent\textcircled{4} The next part of the proof relies on a Taylor expansion of $\nrmng{u_f}p^2$. With no loss of generality, by rotational invariance, we can assume that $\nu=(1,0,\ldots,0)$ so that with Cartesian coordinates $x=(x_1,x_2,\ldots,x_n)\in\R^n$, we write $u_f(x)=1+\varepsilon\,x_1+\eta\,r(x)$. The following result is at the core of our strategy. It heavily relies on the Gaussian logarithmic Sobolev inequality and new estimates for the remainder terms based on the boundedness of $\irng{r^2\,\log r^2}$.
%-------------------------------------------------------------------------
\begin{lemma}\label{Lem:Stab4} Let $n\ge1$ and $f\in\mathrm H^1(\R^n,d\gamma)$ be a non-negative function such that~\eqref{Neighbourhood} holds. We keep the same notation as in Lemma~\ref{Lem:Stab3} and further assume that $\eta\le t\,\varepsilon^2$ for some $t>0$. Then there is a constant $\mathcal C>0$, depending only on $n$, $p$ and $t$, such that
\begin{multline*}
\nrmng{u_f}p^2\ge1+(p-1)\(\varepsilon^2+\eta^2\nrmng r2^2\)\\
+(p-1)\,(2-p)\(\frac12\,\varepsilon^4+\varepsilon^2\,\eta\irng{x_1^2\,r(x)}\)
-\frac{\mathcal C\,\varepsilon^4}{\log\(\frac1\varepsilon\)}\quad\mbox{as}\quad\varepsilon\to0_+\,.
\end{multline*}
\end{lemma}
%-------------------------------------------------------------------------
\begin{proof} This proof is elementary although a little bit lengthy. Several steps are similar to those of~\cite[Section~4]{BDS2022-Sphere} and will be only sketched. Let us split it into three steps.

\medskip\noindent\textit{\textbf{Step 1.}} Let us start with a list of preliminary remarks.

\smallskip\noindent$\bullet$ Let $B_\varepsilon$ be the centred ball of radius $1/(2\,\varepsilon)$, that is,
\[
B_\varepsilon:=\big\{x\in\R^n\,:\,2\,\varepsilon\,|x|<1\big\}
\]
and let $B_\varepsilon^c=\R^n\setminus B_\varepsilon$. We observe that
\[
\gamma(B_\varepsilon^c)=\big|\S^{n-1}\big|\int_{1/(2\,\varepsilon)}^{+\infty}r^{n-1}\,e^{-\frac{r^2}2}\,dr=c_n\,\varepsilon^{2-n}\,e^{-\frac1{8\,\varepsilon^2}}\(1+O\(\varepsilon^2\)\)\quad\mbox{as}\quad\varepsilon\to0_+
\]
with $c_n=2^{3\,(2-n)/2}/\Gamma(n/2)$. Let $\xi_p:=\sup_{\varepsilon\in(0,1/2)}\varepsilon^{-5}\big(\gamma(B_\varepsilon^c)\big)^{(2-p)/2}$. Hence we have
\[
\int_{B_\varepsilon^c}|g|^p\,d\gamma\le\nrmng g2^p\,\big(\gamma(B_\varepsilon^c)\big)^{(2-p)/2}\le\xi_p\,\nrmng g2^p\,\varepsilon^5
\]
for any $g\in\mathrm L^2(\R^n,d\gamma)$, by H\"older's inequality and, as a consequence, 
\be{Outside}
\int_{B_\varepsilon}|g|^p\,d\gamma\le\irng{|g|^p}\le\int_{B_\varepsilon}|g|^p\,d\gamma+\xi_p\,\nrmng g2^p\,\varepsilon^5\quad\forall\,\varepsilon\in(0,1/2)\,.
\ee
{}From now on, we assume without further notice that $x\in B_\varepsilon$ unless it is specified.

\smallskip\noindent$\bullet$ An expansion in Taylor series of $(1+s)^p-1-p\,s$ for $s\le0$ shows that all terms are non-negative:
\[
(1+s)^p\ge1+p\,s+\frac12\,p\,(p-1)\,s^2\quad\forall\,s\in(-1,0]\,.
\]
Applied to $u_f=1+\varepsilon\,x_1+\eta\,r$ whenever $1+\varepsilon\,x_1>0$ and $r\le0$, we obtain
\begin{multline*}
|u_f|^p=\(1+\varepsilon\,x_1+\eta\,r\)^p=\(1+\varepsilon\,x_1\)^p\(1+\frac{\eta\,r}{1+\varepsilon\,x_1}\)^p\\
\ge\(1+\varepsilon\,x_1\)^p\(1+p\,\frac{\eta\,r}{1+\varepsilon\,x_1}+\frac12\,p\,(p-1)\(\frac{\eta\,r}{1+\varepsilon\,x_1}\)^2\)\\
=\(1+\varepsilon\,x_1\)^p+p\(1+\varepsilon\,x_1\)^{p-1}\,\eta\,r+\frac12\,p\,(p-1)\(1+\varepsilon\,x_1\)^{p-2}\,\eta^2\,r^2\,.
\end{multline*}

\smallskip\noindent$\bullet$ Let us consider the case $1+\varepsilon\,x_1>0$ and $r>0$. The function
\[
\rho(s):=\frac1{s^2}\((1+s)^p-1-p\,s-\frac12\,p\,(p-1)\,s^2\)\quad\,\forall\,s\ge0
\]
is bounded. Let us extend $\rho$ by $0$ on $(-1,0)$ and define
\[
\psi_{\varepsilon,\eta,r}(x):=\(1+\varepsilon\,x_1\)^{p-2}\,\rho\(\frac{\eta\,r(x)}{1+\varepsilon\,x_1}\)\,.
\]
With this definition, using $u_f\ge0$ by hypothesis, we obtain
\begin{align*}
|u_f|^p\,&\ge\(1+\varepsilon\,x_1+\eta\,r\)^p\\
&=\(1+\varepsilon\,x_1\)^p+p\(1+\varepsilon\,x_1\)^{p-1}\,\eta\,r+\frac12\,p\,(p-1)\(1+\varepsilon\,x_1\)^{p-2}\,\eta^2\,r^2+\frac12\,p\,(p-1)\,\eta^2\,r(x)^2\,\psi_{\varepsilon,\eta,r}(x)
\end{align*}
with equality whenever $r\ge0$.

As $\eta\to0_+$, $\psi_{\varepsilon,\eta,r}$ converges a.e.~to $0$ on $B_\varepsilon$ uniformly with respect to $\varepsilon\in(0,1/2)$. The dominated convergence theorem is enough to conclude that
\[
\lim_{\eta\to0_+}\int_{B_\varepsilon}r(x)^2\,\psi_{\varepsilon,\eta,r}(x)\,d\gamma=0
\]
for a given function $r$, but this is not enough to conclude uniformly with respect to $r$. To do this, we need more detailed estimates. Notice however that
\[
\|\psi_{\varepsilon,\eta,r}\|_{\mathrm L^\infty(B_\varepsilon)}\le M:=2^{2-p}\,\|\rho\|_{\mathrm L^\infty(\R^+)}
\]
where $M$ is independent of $\varepsilon$, $\eta$ and $r$.

\smallskip\noindent$\bullet$ Since $\int_{B_\varepsilon}x_1^{2\,k}\,d\gamma\ge\int_{B_\varepsilon}x_1^{2\,k+1}\,d\gamma=0$ and $\irng{x_1^{2\,k}}\ge\irng{x_1^{2\,k+1}}=0$ for any $k\in\N$, an expansion in Taylor series of $\(1+\varepsilon\,x_1\)^p$ gives
\begin{multline*}
\irng{\(1+\varepsilon\,x_1\)^p}\ge\irng{\(1+\frac12\,p\,(p-1)\,\varepsilon^2\,x_1^2+\frac1{24}\,p\,(p-1)\,(p-2)\,(p-3)\,x_1^4\,\varepsilon^4\)}\\
=1+\frac12\,p\,(p-1)\,\varepsilon^2+\frac18\,p\,(p-1)\,(p-2)\,(p-3)\,\varepsilon^4\,.
\end{multline*}
By applying~\eqref{Outside} with $g=1+\varepsilon\,x_1$, we obtain
\[
\irng{|1+\varepsilon\,x_1|^p}\le\int_{B_\varepsilon}|1+\varepsilon\,x_1|^p\,d\gamma+\xi_p\(1+\varepsilon^2\)^{p/2}\,\varepsilon^5\,.
\]
Summing up with $\varepsilon^2\le\theta$, we have
\[
\int_{B_\varepsilon}|1+\varepsilon\,x_1|^p\,d\gamma\ge1+\frac12\,p\,(p-1)\,\varepsilon^2+\frac18\,p\,(p-1)\,(p-2)\,(p-3)\,\varepsilon^4-\xi_p\,(1+\theta)^{p/2}\,\varepsilon^5\,.
\]

\smallskip\noindent$\bullet$ Let us estimate $\nrmng{u_f}p^p$ using
\[
\irng{|u_f|^p}\ge\irneps{|u_f|^p}
\]
and
\begin{multline*}
\irneps{|u_f|^p}\ge\irneps{|1+\varepsilon\,x_1|^p}\\
+p\,\eta\irneps{\(1+\varepsilon\,x_1\)^{p-1}\,r}+\frac12\,p\,(p-1)\,\eta^2\irneps{\(1+\varepsilon\,x_1\)^{p-2}\,r^2}\\
+\frac12\,p\,(p-1)\,\eta^2\irneps{r(x)^2\,\psi_{\varepsilon,\eta,r}(x)}\,.
\end{multline*}
We obtain
\begin{multline*}
\irng{|u_f|^p}\ge1+\frac12\,p\,(p-1)\,\varepsilon^2+\frac18\,p\,(p-1)\,(p-2)\,(p-3)\,\varepsilon^4-\xi_p\,(1+\theta)^{p/2}\,\varepsilon^5\\
+p\,\eta\irneps{\(1+\varepsilon\,x_1\)^{p-1}\,r}+\frac12\,p\,(p-1)\,\eta^2\irneps{\(1+\varepsilon\,x_1\)^{p-2}\,r^2}\\
+\frac12\,p\,(p-1)\,\eta^2\irneps{r(x)^2\,\psi_{\varepsilon,\eta,r}(x)}\,.
\end{multline*}

\medskip\noindent\textit{\textbf{Step 2.}} We prove that $\eta^2\irneps{r(x)^2\,\psi_{\varepsilon,\eta,r}(x)}$ is of order $o\big(\varepsilon^4\big)$ as $\eta\le t\,\varepsilon^2\to0$ for a given $t>0$.

\smallskip\noindent$\bullet$ By the logarithmic Sobolev inequality,
\[
\irng{h^2\,\log h^2}\le2\irng{|\nabla h|^2}+\irng{h^2}\,\log\(\irng{h^2}\)\,,
\]
applied to $h=1+(r-1)_+$, we learn that
\[
\irng{h^2\,\log h^2}\le2\irng{|\nabla r|^2}+\(1+\irng{r^2}\)\,\log\(1+\irng{r^2}\)\le2\,\log(2\,e)\,.
\]
Let $\chi:=\mathbb1_{\{\eta\,h>s_0\}}$ for any $s_0>1$ and consider $A_{\varepsilon,\eta,r,s_0}:=\left\{x\in B_\varepsilon\,:\,\eta\,r(x)\le s_0\right\}$. Then we have
\[
\irng{h^2\,\log h^2}\ge\int_{B_\varepsilon\setminus A_{\varepsilon,\eta,r,s_0}}h^2\,\log h^2\,d\gamma=\irneps{\chi\,h^2\,\log h^2}\ge\log\(\frac{s_0}\eta\)^2\irneps{\chi\,r^2}
\]
because $h\ge1$ a.e. and $h>s_0/\eta$ on $B_\varepsilon\setminus A_{\varepsilon,\eta,r,s_0}$ and, as a consequence,
\[
\eta^2\int_{B_\varepsilon\setminus A_{\varepsilon,\eta,r,s_0}}r(x)^2\,\psi_{\varepsilon,\eta,r}(x)\,d\gamma\ge-\,\frac{\log(2\,e)\,M}{\log s_0+\log\(\frac1\eta\)}\,\eta^2\,.
\]

\smallskip\noindent$\bullet$ Let us notice that
\[
\mathcal M=\sup_{s>0}\frac{|\rho(s)|}{\log(1+s)}
\]
is finite. This allows us to write that
\[
\eta^2\int_{A_{\varepsilon,\eta,r,s_0}}r(x)^2\,\psi_{\varepsilon,\eta,r}(x)\,d\gamma\ge-\,2^{2-p}\,\mathcal M\int_{A_{\varepsilon,\eta,r,s_0}}\eta^2\,r^2\,\log(1+2\,\eta\,r)\,\mathbb 1_{\{r>0\}}\,d\gamma\,,
\]
where the restriction to the set $\{r>0\}$ comes from the fact that $\psi_{\varepsilon,\eta,r}(x)=0$ whenever $r(x)\le0$. Now we estimate $\log(1+2\,\eta\,r)$ by
\begin{align*}
&\log(1+2\,\eta\,r)\le\log\(1+2\,\sqrt\eta\)\quad&\mbox{if}\quad0\le r\le\frac1{\sqrt\eta}\,,\\
&\log(1+2\,\eta\,r)\le\frac{\log(1+2\,s_0)}{\log\(\frac1\eta\)}\,\log r^2\quad&\mbox{if}\quad\frac1{\sqrt\eta}\le r\le\frac{s_0}\eta\,,
\end{align*}
and conclude using $\irng{r^2}\le1$ and $\irng{r^2\,\log r^2}\le2\irng{|\nabla r|^2}=2$ by the logarithmic Sobolev inequality that
\[
\eta^2\int_{A_{\varepsilon,\eta,r,s_0}}r(x)^2\,\psi_{\varepsilon,\eta,r}(x)\,d\gamma\ge-\,2^{2-p}\,\mathcal M\,\eta^2\(\theta\,\log\(1+2\,\sqrt\eta\)+\frac{2\,\log(1+2\,s_0)}{\log\(\frac1\eta\)}\)\,.
\]

\medskip\noindent\textit{\textbf{Step 3.}} We compute the contribution of
\[
p\,\eta\irneps{\(1+\varepsilon\,x_1\)^{p-1}\,r}\quad\mbox{and}\quad\frac12\,p\,(p-1)\,\eta^2\irneps{\(1+\varepsilon\,x_1\)^{p-2}\,r^2}
\]
to the expansion of $\irng{|u_f|^p}$.

\smallskip\noindent$\bullet$ Using~\eqref{Outside} applied to $g=|r|^{1/p}$ and the orthogonality constraints on $r$, we obtain
\begin{multline*}
p\,\eta\int_{B_\varepsilon}(1+\varepsilon\,x_1)^{p-1}\,r\,d\gamma\ge p\,\eta\(\irng{(1+\varepsilon\,x_1)^{p-1}\,r}-2^{1-p}\,\xi_p\,\varepsilon^5\)\\
\ge p\,\eta\,\varepsilon^2\(\frac12\,(p-1)\,(p-2)\irng{x_1^2\,r}-\mathsf c_1\,\varepsilon^3-2^{1-p}\,\xi_p\,\varepsilon^5\)
\end{multline*}
where $\mathsf c_1:=\sqrt{15}\,\sup_{s\in(-1,0)\cup(0,+\infty)}\left|(1+s)^{p-1}-1-(p-1)\,s-\frac12\,(p-1)\,(p-2)\,s^2\right|/s^3$

\smallskip\noindent$\bullet$ A similar computation shows that
\[
\frac12\,p\,(p-1)\,\eta^2\irneps{\(1+\varepsilon\,x_1\)^{p-2}\,r^2}=\frac12\,p\,(p-1)\,\eta^2\(\nrmng r2^2-\mathsf c_2\,\varepsilon\)
\]
where $\mathsf c_2:=\sup_{s\in(-1,0)\cup(0,+\infty)}\left|(1+s)^{p-2}-1\right|/s$

\medskip\noindent\textit{\textbf{Step 4.}} Collecting all terms, we have

\begin{multline*}
\irng{|u_f|^p}\ge1+\frac12\,p\,(p-1)\,\varepsilon^2+\frac18\,p\,(p-1)\,(p-2)\,(p-3)\,\varepsilon^4+\frac12\,p\,(p-1)\,\eta^2\,\nrmng r2^2\\
+\frac12\,p\,(p-1)\,(p-2)\,\eta\,\varepsilon^2\,\irng{x_1^2\,r}-\,C\,\frac{\varepsilon^4}{\log\(\frac1\varepsilon\)}
\end{multline*}
for some constant $C>0$ that is explicitly given in terms of $t$, $\xi_p$, $M$, $\mathcal M$, $\mathsf c_1$ and $\mathsf c_2$. In order to conclude, we notice that for any $s>-1$, $\frac{d^4}{ds^4}(1+s)^\frac2p>0$ implies that
\[
(1+s)^\frac2p\ge1+\frac2p\,s+\frac1{p^2}\,(2-p)\,s^2-\frac1{3\,p^3}\,(2-p)\,(4-p)\,s^3\,.
\]
Applied to
\[
\nrmng{u_f}p^2=\(\irng{|u_f|^p}\)^\frac2p=\(\irng{\(1+\varepsilon\,x_1+\eta\,r\)^p}\)^\frac2p\,,
\]
this completes the proof of Lemma~\ref{Lem:Stab4}.
\end{proof}

\begin{proof}[Proof of Theorem~\ref{stability-subcritical-Gaussian}] The strategy of proof is similar to~\cite[Theorem 7]{BDS2022-Sphere}. Up to the replacement of~$v$ by $f=|v|$, we can assume $f\ge0$. With
\[
u_f(x):=\frac{f(x)}{\irng f}=1+\varepsilon\,x_1+\eta\,r(x)
\]
as in Lemma~\ref{Lem:Stab3}, we first notice that the case $\eta>t\,\varepsilon^2$ is already covered in Lemma~\ref{Lem:Stab3}. Otherwise let us assume that $\eta\le t\,\varepsilon^2$ and consider the Hermite polynomial $h_1(x):=x_1^2-1$ and decompose $r$ according~to
\[
r(x):=\alpha\,h_1(x)+\beta\,\tilde r(x)
\]
with $\nrmng{\tilde r}2=1$ so that $\nrmng r2^2=\alpha^2+\beta^2\le1/2$ and $\irng{x_1^2\,r(x)}=\alpha$. With this notation we have
\begin{align*}
&\nrmng{\nabla u_f}2^2=\varepsilon^2+\eta^2\,,\\ 
&\nrmng{u_f}2^2=1+\varepsilon^2+\eta^2\(\alpha^2+\beta^2\)\,,\\
&\nrmng{u_f}p^2\ge1+(p-1)\(\varepsilon^2+\eta^2\(\alpha^2+\beta^2\)\)+(p-1)\,(2-p)\(\frac12\,\varepsilon^4+\alpha\,\varepsilon^2\,\eta\)-\frac{\mathcal C\,\varepsilon^4}{\log\(\frac1\varepsilon\)}\,,
\end{align*}
where the estimate of $\nrmng{u_f}p^2$ comes from Lemma~\ref{Lem:Stab4}. Hence, for some $\lambda>0$ to be fixed, 
\begin{align*}
&\nrmng{\nabla u_f}2^2-\frac1{2-p}\(\nrmng{u_f}2^2-\nrmng{u_f}p^2\)-\lambda\(\eta^2+\frac{t^2\,\varepsilon^4}{1+\varepsilon^2+\eta^2}\)\\
&\ge\eta^2\(1-\alpha^2-\beta^2-\lambda\)+\(\frac{p-1}2-\frac{\lambda\,t^2}{1+\varepsilon^2+\eta^2}\)\varepsilon^4+(p-1)\,\alpha\,\varepsilon^2\,\eta-\frac1{2-p}\,\frac{\mathcal C\,\varepsilon^4}{\log\(\frac1\varepsilon\)}\\
&\ge \(\frac{p-1}2-\frac{\lambda\,t^2}{1+\varepsilon^2+\eta^2} - \frac{\alpha^2 \, (p-1)^2}{ 4 \, \(1 - \alpha^2 - \beta^2 -2 \lambda\right)}\)\varepsilon^4 - \frac1{2-p}\,\frac{\mathcal C\,\varepsilon^4}{\log\(\frac1\varepsilon\)}
\end{align*}
where we applied Young's inequality in the last line. By an appropriate choice of the parameters, for instance $\lambda\in(0,1/16)$ and $t>0$ such that $t^2<8\,(p-1)\,(5-2\,p)/3$, we obtain a positive coefficient in front of $\varepsilon^4$, and the last line is non-negative if $\varepsilon>0$ is taken small enough. Hence, the result with $c_{n,p}=\min\big\{c_{n,p}^{(1)},c_{n,p}^{(2)}\big\}$ with $c_{n,p}^{(1)}$ given by~\eqref{cnp1} and $c_{n,p}^{(2)}:=(1-\theta)\,\lambda$.
\end{proof}

%%%%%%%%%%%%%%%%%%%%%%%%%%%%%%%%%%%%%%%%%%%%%%%%%%%%%%%%%%%%%%%%%%%%%%%%%%
%%%%%%%%%%%%%%%%%%%%%%%%%%%%%%%%%%%%%%%%%%%%%%%%%%%%%%%%%%%%%%%%%%%%%%%%%%
\section*{Acknowledgements}
{\small G.B.~has been funded by the European Union's Horizon 2020 research and innovation program under the Marie Sklodow\-ska-Curie grant agreement No.~754362. The authors thank M.J.~Esteban for useful comments. The authors warmfully thank a referee who replied very quickly with a detailed report and a list of suggestions which greatly improved the revised version of this paper.}\par
\smallskip\noindent{\scriptsize\copyright\,2023 by the authors. This paper may be reproduced, in its entirety, for non-commercial purposes.}
%%%%%%%%%%%%%%%%%%%%%%%%%%%%%%%%%%%%%%%%%%%%%%%%%%%%%%%%%%%%%%%%%%%%%%%%%%
%%%%%%%%%%%%%%%%%%%%%%%%%%%%%%%%%%%%%%%%%%%%%%%%%%%%%%%%%%%%%%%%%%%%%%%%%%
\bibliographystyle{crplain}
\bibliography{BDS2023-PLS}
%\bigskip\begin{center}\rule{2cm}{0.5pt}\end{center}\bigskip
%\tableofcontents
\end{document}
%%%%%%%%%%%%%%%%%%%%%%%%%%%%%%%%%%%%%%%%%%%%%%%%%%%%%%%%%%%%%%%%%%%%%%%%%%
%%%%%%%%%%%%%%%%%%%%%%%%%%%%%%%%%%%%%%%%%%%%%%%%%%%%%%%%%%%%%%%%%%%%%%%%%%